\documentclass [11pt,oneside,a4paper,mathscr]{amsart}

\usepackage{amscd,amsmath,amssymb,euscript}
\usepackage[frame,cmtip,curve,arrow,matrix,line,graph]{xy}
\usepackage{tikz-cd}
\usepackage{bigints}
\usetikzlibrary{matrix}
\usepackage{hyperref}
\usepackage{stmaryrd}

\title{Intersection $K$-theory}
\author{Tudor P\u adurariu}
\address{
Department of Mathematics, Columbia University, 
2990 Broadway, New York, NY 10027}
\email{tgp2109@columbia.edu}

\newtheorem{thm}{Theorem}[section]
\newtheorem{cor}[thm]{Corollary}
\newtheorem{conj}[thm]{Conjecture}
\newtheorem{prop}[thm]{Proposition}

\theoremstyle{definition}

\newtheorem{thm*}[thm]{Theorem$^*$}

\newcommand{\comment}[1]{}

\renewcommand{\leq}{\leqslant}
\renewcommand{\geq}{\geqslant}

\newcommand{\KK}{\text{gr}^\cdot K_\cdot(X)}

\newcommand{\OO}{\mathcal{O}}
\renewcommand{\P}{\textbf{P}}
\newcommand{\E}{\mathcal{E}}
\newcommand{\I}{\textbf{I}}

\newcommand{\C}{\mathbb{C}}

\begin{document}
\maketitle

\begin{abstract}
For a proper map $f:X\to S$ between varieties over $\mathbb{C}$ with $X$ smooth, we introduce increasing filtrations $\P^{\leq \cdot}_f\subset P^{\leq \cdot}_f$ on $\text{gr}^\cdot K_\cdot(X)$, the associated graded on $K$-theory with respect to the codimension filtration, both sent by the cycle map to the perverse filtration on cohomology ${}^pH^{\leq \cdot}_f(X)$. The filtrations $P^{\leq \cdot}_f$ and $\P^{\leq \cdot}_f$ are functorial with respect to proper pushforward and $P^{\leq \cdot}_f$ is functorial with respect to pullback.

We use the above filtrations to propose two definitions of (graded) intersection $K$-theory $\text{gr}^\cdot IK_\cdot(S)$ and $\text{gr}^\cdot \textbf{I}K_\cdot(S)$. Both have cycle maps to intersection cohomology $IH^{\cdot}(S)$. 
We conjecture a version of the decomposition theorem for semismall surjective maps and prove it in some particular cases.
\end{abstract}


\section{Introduction}


For a complex variety $X$, intersection cohomology $IH^\cdot(X)$ coincides with singular cohomology with rational coefficients $H^\cdot(X)$ when $X$ is smooth and has better properties than $H^\cdot(X)$ when $X$ is singular, for example it satisfies Poincar\'e duality and the Hard Lefschetz theorem.  
Many applications of intersection cohomology, for example in representation theory \cite{L}, \cite[Section 4]{dCM2}, are through the decomposition theorem of Beilinson--Bernstein--Deligne--Gabber \cite{BBD}. 

A construction of intersection $K$-theory is expected to have applications in computations of $K$-theory via a $K$-theoretic version of the decomposition theory, and in representation theory, for example in the construction of representations of vertex algebras using (framed) Uhlenbeck spaces \cite{BFN}. 
The Goresky--MacPherson construction of intersection cohomology \cite{GMcP2} does not generalize in an obvious way to $K$-theory.

\subsection{The perverse filtration and intersection cohomology}
\label{coh}

For $S$ a variety over $\mathbb{C}$, intersection cohomology $IH^\cdot(S)$ is a subquotient of $H^\cdot(X)$ for any resolution of singularities $f:X\to S$.
The decomposition theorem implies that $IH^{\cdot}(S)$ is a (non-canonical) direct summand of $H^{\cdot}(X)$. Consider the perverse filtration
\[{}^pH^{\leq i}_f(X):=H^{\cdot}\left(S, {}^p\tau^{\leq i}Rf_*IC_X\right)\hookrightarrow H^{\cdot}(S, Rf_*IC_X)=H^{\cdot}(X).\]
For $V\hookrightarrow S$, denote by $X_V:=f^{-1}(V)$. Let $A_V$ be the set of irreducible components of $X_V$ and let $c^a_V$ be the codimension on $\iota^a_V: X_V^a\hookrightarrow X$ for $a\in A_V$. Consider a resolution of singularities $\pi^a_V: Y_V^a\to X_V^a$.
Let $g^a_V:=f\pi^a_V:Y_V^a\to V$. Define 
\begin{align*}
    {}^p\widetilde{H}_{f,V}^{\leq i}:=&\bigoplus_{a\in A_V}\iota^a_{V*}\pi_{V*}^a\, {}^pH^{\leq i-c^a_V}_{g^a_V}(Y_V^a)\subset {}^pH_f^{\leq i}(X),\\
     {}^p\widetilde{H}_{f}^{\leq i}:=&\bigoplus_{V\subsetneqq S}{}^p\widetilde{H}_{f,V}^{\leq i}\subset {}^pH_f^{\leq i}(X).
\end{align*}
The decomposition theorem implies that
\[IH^{\cdot}(S)\cong{}^pH_f^{\leq 0}H^\cdot(X)\big/{}^p\widetilde{H}_f^{\leq 0}H^\cdot(X).\]

\subsection{Perverse filtrations in $K$-theory}

Inspired by the above characterization of intersection cohomology via the perverse filtration, we propose two $K$-theoretic perverse filtrations $\P^{\leq i}_f\subset P^{\leq i}_f$ on $\text{gr}^\cdot K_\cdot(X)$ for a proper map $f:X\to S$ of complex varieties with $X$ smooth. Here, the associated graded $\text{gr}^\cdot K_\cdot(X)$ is with respect to the codimension of support filtration on $K_\cdot(X)$ \cite[Definition 3.7, Section 5.4]{G}.

The precise definition of the filtration $P^{\leq i}_f \text{gr}^\cdot K_\cdot(X)$ is given in Subsection \ref{filtPP}; roughly, it is generated by (subspaces of) images
\begin{equation}\label{correspo}
\Phi_\Gamma: \text{gr}^\cdot K_\cdot(T)\to \text{gr}^\cdot K_\cdot(X)
\end{equation}
induced by correspondences $\Gamma$ on $X\times T$ of restricted dimension, see \eqref{ranges}, for $T$ a smooth variety with a generically finite map onto a subvariety of $S$. 
These subspaces of the images of $\Phi_\Gamma$ are required to satisfy certain conditions when restricted to the subvarieties $Y^a_V$ from Subsection \ref{coh}. 

The definition of the filtration $\P^{\leq i}_f \text{gr}^\cdot K_\cdot(X)\subset P^{\leq i}_f \text{gr}^\cdot K_\cdot(X)$ is given in Subsection \ref{Kper2}. In \eqref{correspo}, we further impose that $\Gamma$ is a quasi-smooth scheme surjective over $T$. This futher restricts the possible dimension of the cycles $\Gamma$, see Proposition \ref{lb}, and allows for more computations.


\begin{thm}\label{cycle}
Let $f:X\to S$ be a proper map with $X$ smooth. Then the cycle map $\mathfrak{c}:\text{gr}^\cdot K_0(X)_\mathbb{Q}\to H^\cdot(X)$ respects the perverse filtration
\[\P^{\leq i}_f\text{gr}^\cdot K_0(X)_\mathbb{Q}\subset P^{\leq i}_f\text{gr}^\cdot K_0(X)_\mathbb{Q}\xrightarrow{\mathfrak{c}} {}^pH^{\leq i}_f(X).\]
\end{thm}

Perverse filtrations in $K$-theory have the following functorial properties. Let $X$ and $Y$ be smooth varieties with $c=\dim X-\dim Y$. Consider proper maps
\begin{equation*}
    \begin{tikzcd}[column sep=small]
    Y\arrow[dr,"g"']\arrow[rr,"h"] && X\arrow[dl,"f"]\\
    & S. &
    \end{tikzcd}
\end{equation*}
There are induced maps
\begin{align*}
h_* &:P_g^{\leq i-c}\text{gr}_\cdot K_\cdot(Y)\to P_f^{\leq i}\text{gr}_\cdot K_\cdot(X),\\
h_* &:\P_g^{\leq i-c}\text{gr}_\cdot K_\cdot(Y)\to \P_f^{\leq i}\text{gr}_\cdot K_\cdot(X),\\
h^* &:P_f^{\leq i-c}\text{gr}^\cdot K_\cdot(X)\to P_g^{\leq i}\text{gr}^\cdot K_\cdot(Y).
\end{align*}
If $h$ is surjective, then there is also a map
\[h^*:\P_f^{\leq i-c}\text{gr}^\cdot K_\cdot(X)\to \P_g^{\leq i}\text{gr}^\cdot K_\cdot(Y).\]

Let $f:X\to S$ be a resolution of singularities. We define $\widetilde{P}^{\leq 0}_f\text{gr}^\cdot K_\cdot(X)$ and $\widetilde{\P}^{\leq 0}_f\text{gr}^\cdot K_\cdot(X)$ similarly to ${}^p\widetilde{H}^{\leq i}(X)$. 
Inspired by the discussion in cohomology from Subsection \ref{coh}, define
\begin{align*}
    \text{gr}^\cdot IK_{\cdot}(S)&:=P_f^{\leq 0}\text{gr}^\cdot K_\cdot(X)\big/\left(\widetilde{P}_f^{\leq 0}\text{gr}^\cdot K_\cdot(X)\cap \text{ker}\,f_*\right)
\\
\text{gr}^\cdot \textbf{I}K_{\cdot}(S)&:=\P_f^{\leq 0}\text{gr}^\cdot K_\cdot(X)\big/\left(\widetilde{\P}_f^{\leq 0}\text{gr}^\cdot K_\cdot(X)\cap \text{ker}\,f_*\right).\end{align*}
Note that we do not construct $\text{gr}^\cdot IK_{\cdot}(S)$ and 
$\text{gr}^\cdot \textbf{I}K_{\cdot}(S)$ as associated graded of spaces $IK_{\cdot}(S)$ or $\textbf{I}K_{\cdot}(S)$, but we hope that such a construction is possible, see Subsection \ref{inkt}.

\begin{thm}\label{well}
The definitions of $\text{gr}^\cdot IK_{\cdot}(S)$ and $\text{gr}^\cdot \textbf{I}K_{\cdot}(S)$ do not depend on the resolution of singularities $f:X\to S$ with the properties mentioned above. 
Further, there are cycle maps
\begin{align*}
\mathfrak{c}&: \text{gr}^jIK_0(S)_\mathbb{Q}\to IH^{2j}(S)\\
\mathfrak{c}&: \text{gr}^j\textbf{I}K_0(S)_\mathbb{Q}\to IH^{2j}(S).
\end{align*}
\end{thm}

We also propose definitions for $\text{gr}^\cdot IK_{\cdot}(S,L)$ and $\text{gr}^\cdot \textbf{I}K_{\cdot}(S,L)$ for $L$ a local system on $U$ open in $S$ of the form $L\cong h_*\left(\mathbb{Z}_V\right)$ for an étale map $h:V\to U$. There are cycle maps
\begin{align*}
\mathfrak{c}&: \text{gr}^jIK_0(S, L)_\mathbb{Q}\to IH^{2j}(S, L\otimes\mathbb{Q})\\
\mathfrak{c}&: \text{gr}^j\textbf{I}K_0(S, L)_\mathbb{Q}\to IH^{2j}(S, L\otimes\mathbb{Q}).
\end{align*}

\subsection{Properties of the perverse filtrations and intersection $K$-theory}

The perverse filtrations in $K$-theory and intersection $K$-theory have similar properties to their counterparts in cohomology.

For a map $f:X\to S$, let
$s:=\dim X\times_SX-\dim X$ be its defect of semismallness. In Theorem \ref{sd}, we show that \begin{align*}
\P^{\leq -s-1}_f\text{gr}^\cdot K_0(X)&=0,\\
\P^{\leq s}_f\text{gr}^\cdot K_0(X)=P^{\leq s}_f\text{gr}^\cdot K_0(X)&=\text{gr}^\cdot K_0(X).
\end{align*}
This implies that
\begin{align*}
    & \text{gr}^\cdot IK_\cdot(S)=
    \text{gr}^\cdot \textbf{I}K_\cdot(S)=
    \text{gr}^\cdot K_\cdot(S) \text{ for } S \text{ smooth,}\\
    & \text{gr}^\cdot \textbf{I}K_0(S)=\text{gr}^\cdot K_0(X) \text{ if } S \text{ has a small resolution }f:X\to S.
\end{align*}
For more computations of perverse filtrations in $K$-theory and intersection $K$-theory, see Subsections \ref{examplesfilt} and \ref{computations}.

Let $d=\dim S$. In cohomology, there are natural maps
\begin{align*}
&H^i(S)\to IH^i(S)\to H^{\text{BM}}_{2d-i}(S)\\
&IH^i(S)\otimes IH^j(S)\to H^{\text{BM}}_{2d-i-j}(S).
    \end{align*}
    The composition in the first line is the natural map $H^i(S)\to H^{\text{BM}}_{2d-i}(S)$. The second map is non-degenerate for cycles of complementary dimensions.
In Subsection \ref{multK} we explain that there exist natural maps
\begin{align*}
\text{gr}_i IK_\cdot(S)&\to \text{gr}_i G(S)\\
    \text{gr}^i IK_\cdot(S)\times \text{gr}^j IK_\cdot(S)&\to \text{gr}_{d-i-j} G_\cdot(S)
    \end{align*}
    and their analogues for $\I K$.
The above filtration on $G$-theory is by dimension of supports, see \cite[Section 5.4]{G}. 

\subsection{The decomposition theorem for semismall maps}\label{ssn}
As mentioned above, many applications of intersection cohomology are based on the decomposition theorem. When the map 
\[f:X\to S\] is semismall, the statement of the decomposition theorem is more explicit, which we now explain. Let $\{S_a|\,a\in I\}$ be a stratification of $S$ such that $f_a: f^{-1}(S_a^o)\to S_a^o$ is a locally trivial fibration, where $S_a^o=S_a-\bigcup_{b\in I} \left(S_a\cap S_b\right).$
Let $A\subset I$ be the set of relevant strata, that is, those strata such that for $x_a\in S_a^o$: 
\[\dim f^{-1}(x_a)=\frac{1}{2}\left(\dim S-\dim S_a\right).\]
For $x_a\in S_a^o$, the monodromy group $\pi_1(S_a^0, x_a)$ acts on the set of irreducible components of $f^{-1}(x_a)$ of top dimension; let $L_a$ be the corresponding local system.
Let $c_a$ be the codimension of $X_a=f^{-1}(S_a)$ in $X$. The decomposition theorem for the map $f:X\to S$ says that there exists a canonical decomposition \cite[Theorem 4.2.7]{dCM2}:
\[H^j(X)\cong\bigoplus_{a\in A} IH^{j-c_a}(S_a, L_a).\]

We conjecture the analogous statement in $K$-theory.

\begin{conj}\label{dss}
Let $f:X\to S$ be a semismall map and consider $\{S_a|\,a\in I\}$ a stratification as above, and let $A\subset I$ be the set of relevant strata. There is a decomposition for any integer $j$:
\[\text{gr}^j K_\cdot(X)_{\mathbb{Q}}\cong\bigoplus_{a\in A} \text{gr}^{j-c_a} \I K_\cdot(S_a, L_a)_{\mathbb{Q}}.\]
\end{conj}

See Conjecture \ref{dss2} for a more precise statement. In Theorem \ref{localrat}, 
we check the above conjecture for $K_0$ under the extra condition that for any $a\in A$, there are small maps $\pi_a:T_a\to S_a$ such that $\pi_a^{-1}(S_a^o)\to S_a^o$ is étale with associated local system $L_a$. The proof of the above result is based on a theorem of de Cataldo--Migliorini \cite[Section 4]{dCM}. In Subsection \ref{surface}, we prove the statement for $K_0$ when $f: X\to S$ is a resolution of singularities of a surface.

\subsection{Intersection Chow groups}

When $X$ is smooth, $\text{gr}^i K_0(X)_{\mathbb{Q}}=CH^i(X)_{\mathbb{Q}}$. Thus $\text{gr}^i IK_0(S)_{\mathbb{Q}}$ is a candidate for an intersection Chow group of $S$. Corti--Hanamura already defined (rational) intersection Chow groups (or Chow motives) in \cite{CH1}, \cite{CH2} inspired by the decomposition theorem. One proposed definition assumes conjectures of Grothendieck and Murre and proves a version of the decomposition theorem for Chow groups; the other approach defines a perverse-type filtration on Chow groups by induction on level $i$ of the filtration and via correspondences involving all varieties $W\to S$.
These correspondences need to satisfy certain vanishings for the perverse filtration in cohomology. The advantage of the definition we propose is that one uses fewer correspondences to define $P^{\leq i}_f$ and $\P^{\leq i}_f$ and this allows for computations, see Subsection \ref{computations} and Theorem \ref{ss}.

For varieties $S$ with a semismall resolution of singularities $f:X\to S$, de Cataldo--Migliorini \cite{dCM} proposed a definition of Chow motives $ICH\,(S)$ and proved a version of the decomposition theorem for semismall maps. 

\subsection{Intersection $K$-theory}\label{inkt}
Our approach to define (graded) intersection $K$-theory uses functoriality of the perverse filtration in an essential way. To obtain this functoriality, it is essential to pass to $\text{gr}^\cdot K_\cdot(X)$. 
It is a very interesting problem to find a definition of the perverse filtration on $K_\cdot(X)$.
We hope that such a definition will provide a version of equivariant intersection $K$-theory with applications to geometric representation theory, for example in understanding the $K$-theoretic version of \cite{BFN}.

There are proposed definitions of intersection $K$-theory in some particular cases. Cautis \cite{C}, Cautis--Kamnitzer \cite{CK} have an approach for categorification of intersection sheaves for certain subvarieties of the affine Grassmannian.  Eberhardt defined intersection $K$-theoretic sheaves for varieties with certain stratifications \cite{E}.
In \cite{P1}, we proposed a definition of intersection $K$-theoretic for good moduli spaces of smooth Artin stacks which has applications to the structure theory of Hall algebras of Kontsevich--Soibelman \cite{P2}. 



\subsection{Outline of the paper} In Section \ref{2}, we discuss preliminary material. In Section \ref{3}, we introduce the two perverse filtrations $\P^{\leq \cdot}_f\subset P^{\leq \cdot}_f$. In Section \ref{4}, we define the two versions of intersection $K$-theory $\text{gr}^\cdot IK_\cdot(S)$ and $\text{gr}^\cdot \textbf{I}K_\cdot(S)$. In Section \ref{semismall}, we discuss the decomposition theorem for semismall maps. 

\subsection{Acknowledgements}
I thank Jacob Lurie for useful discussions related to the paper. 
I thank the Institute of Advanced Studies for support during the preparation of the paper. This material is based upon work supported by the National Science Foundation under Grant No. DMS-1926686.

\section{Preliminary material}\label{2}

\subsection{Notations and conventions}
All schemes considered in this paper are finite type and defined over $\mathbb{C}$. 
The definition of the filtration in Subsection \ref{Kper} works over any field, but to define intersection $K$-theory we use resolution of singularities. The definition of intersection $K$-theory works over any field of characteristic zero. A variety is an irreducible reduced scheme. We use quasi-smooth schemes in Subsections \ref{filtPP} and \ref{towa}, which are the natural extension of lci schemes in derived geometry.

For $S$ a scheme, let $D^b\text{Coh}(S)$ be the derived category of bounded complexes of coherent sheaves and let
$\text{Perf}(S)$ be its subcategory of bounded complexes of locally free sheaves on $S$. For their analogues for quasi-smooth schemes, see \cite[Subsections 2.1 and 3.1]{K}. The functors used in the paper are derived; we sometimes drop $R$ or $L$ from notation, for example we write $f_*$ instead of $Rf_*$.
When $S$ is smooth, the two categories coincide. Define
\begin{align*}
    G_\cdot(S)&:=K_\cdot\left(D^b\text{Coh}(S)\right),\\
    K_\cdot(S)&:=K_\cdot\left(\text{Perf}(S)\right).
\end{align*}
For $Y$ a subvariety of $X$, let $D^b\text{Coh}_Y(X)$ be the subcategory of $D^b\text{Coh}(X)$ of complexes supported on $Y$, and 
define \[G_{Y, \cdot}(X):=K_\cdot\left(D^b\text{Coh}_Y(X)\right).\] When $X$ is smooth, we also use the notation $K_{Y, \cdot}(X)$ for the above. We will usually drop the subscript $\cdot$ from the notation.

The local systems used in this paper are of the form $L\cong h_*\left(\mathbb{Z}_V\right)$ (or $L\cong h_*\left(\mathbb{Q}_V\right)$) for an étale morphism $h:V\to U$. We call these local systems \textit{integer (or rational) finite local systems}.

Singular and intersection cohomology and Borel-Moore homology are used only with coefficients in a rational finite local system, usually $\mathbb{Q}$.

For $f$ a morphism of varieties, we denote by $\mathbb{L}_f$ its cotangent complex.

\subsection{Filtrations in $K$-theory}
\label{fil}
A reference for the following is \cite{G}, especially Section 5 in loc. cit. Let $F^{i}G_\cdot(S)$ be the filtration on $G_\cdot(S)$ by sheaves with support of codimension $\geq i$; it induces a filtration on $K_\cdot(S)$. The associated graded will be denoted by $\text{gr}^\cdot G_\cdot(S), \text{gr}^\cdot K_\cdot(S)$. A morphism $f:X\to Y$ of smooth varieties induces maps:
\begin{align*}
   f^*&: F^{ i} K_\cdot(Y)\to F^{ i} K_\cdot(X)\\
   f^*&: \text{gr}^iK_\cdot(Y)\to \text{gr}^iK_\cdot(X).
\end{align*}
Further, let $F^{\text{dim}}_iG_\cdot(S)$ be the filtration on $G_\cdot(S)$ by sheaves with support of dimension $\leq i$; it induces a filtration on $K_\cdot(S)$. The associated graded will be denoted by $\text{gr}_\cdot G_\cdot(S), \text{gr}_\cdot K_\cdot(S)$. A proper morphism $f:X\to Y$ of schemes induces maps: 
\begin{align*}
f_*&: F^{\text{dim}}_iG_\cdot(X)\to F^{\text{dim}}_iG_\cdot(Y)\\
f_*&: \text{gr}_iG_\cdot(X)\to \text{gr}_iG_\cdot(Y).
\end{align*}
There are similar filtrations and associated graded on $G_Y(X)$ for $Y\hookrightarrow X$ a subvariety. If $X$ is smooth of dimension $d$, then $\text{gr}_i G_Y(X)=\text{gr}^{d-i}G_Y(X)$.

\begin{prop}\label{prop1}
Let $S\xrightarrow{a} \text{Spec}\,\mathbb{C}$ be a variety of dimension $d$. Then
\[\left(a^*,\bigoplus_{T\subsetneqq S}\iota_{T*}\right): G_0(\text{Spec}\,\mathbb{C})\oplus\bigoplus_{T\subsetneqq S}\text{gr}_\cdot G_0(T)\twoheadrightarrow \text{gr}_\cdot G_0(S),\]
where the sum is taken over all proper subvarieties $T$ of $S$.
\end{prop}

\begin{proof}
For $i<d$, the map 
\[\bigoplus_{T\subsetneqq S}\iota_{T*}: \bigoplus_{T\subsetneqq S}\text{gr}_i G_0(T)\twoheadrightarrow \text{gr}_i G_0(S)\] is surjective by definition of the filtration $F^i_{\text{dim}}$. Finally, the following map is an isomorphism \[a^*: G_0(\text{Spec}\,\mathbb{C})\xrightarrow{\sim}  \text{gr}_dG_0(S).\] 
\end{proof}

\begin{prop}\label{prop2}
Let $S$ be a singular variety of dimension $d$, and let $f:X\to S$ be a resolution of singularities. The following map is surjective:
\[f_*:\text{gr}_i G_0(X)\twoheadrightarrow \text{gr}_i G_0(S).\]
\end{prop}

\begin{proof}
We use induction on $d$. By Proposition \ref{prop1}, the following is an isomorphism \[f_*:\text{gr}_d G_0(X)\xrightarrow{\sim} \text{gr}_d G_0(S)\xrightarrow{\sim} G_0(\text{Spec}\,\C).\]  
For $V\subsetneqq S$ a subvariety, consider $g$ a resolution of singularities as follows:
\begin{equation*}
    \begin{tikzcd}
    Y\arrow[d,"g"]\arrow[r]&X\arrow[d,"f"]\\
    V\arrow[r, hook]&S.
    \end{tikzcd}
\end{equation*}
The surjectivity of $f_*$ for $i<d$ follows using Proposition \ref{prop1} and the induction hypothesis.
\end{proof}

\subsection{Quasi-smooth schemes}\label{qss}
\subsubsection{} A morphism $f:X\to Y$ of derived schemes is \textit{quasi-smooth} if it is locally of finite presentation and the cotangent complex $\mathbb{L}_f$ has Tor-amplitude $\leq 1$. Alternatively, there is a factorization \[X\xrightarrow{\iota} X'\xrightarrow{\pi}Y\] with $\pi$ smooth and $\iota$ a quasi-smooth immersion, that is, the complex $\mathbb{L}_\iota[-1]$ is a vector bundle, see \cite[Section 2]{KR}, \cite[Subsections 2.1 and 2.2]{AG}. A proper map between smooth varieties is quasi-smooth. Any quasi-smooth map has a well defined relative dimension.

A derived scheme $X$ is quasi-smooth if the structure morphism $X\to \text{Spec}\,\mathbb{C}$ is quasi-smooth. Any quasi-smooth scheme has a well defined dimension.

\subsubsection{} We list some properties satisfied by quasi-smooth morphisms and schemes. We are using them only in the proof of an excess intersection formula, see \eqref{eul}. Consider a cartesian diagram
\begin{equation*}
    \begin{tikzcd}
     X'\arrow[d]\arrow[r,"f'"]&Y'\arrow[d]\\
     X\arrow[r,"f"]&Y.
    \end{tikzcd}
\end{equation*}
If $f$ is quasi-smooth, then $f'$ is quasi-smooth and \begin{equation}\label{reldimeq}
    \text{reldim}\,(f)=\text{reldim}\,(f').
\end{equation} Composition of quasi-smooth maps is quasi-smooth. 

Let $Y$ be a quasi-smooth scheme. Define a filtration $F^iG_\cdot(Y)\subset G_\cdot(Y)$ generated by images of \[f_*: G_\cdot(Y)\to G_\cdot(X)\] for $f:Y\to X$ a quasi-smooth morphism with $\text{reldim}\,(f)\leq -i$. It induces a filtration on $K_\cdot(Y)$. One can define similarly a filtration $F^{\text{dim}}_iG_\cdot(Y)$. When $Y$ is a classical scheme, this definitions recover the filtrations introduced in Subsection \ref{fil}, see \cite[Theorem 6.21]{K}. By \eqref{reldimeq}, pullback respects the filtrations $F^i$. Pushforward clearly respects the filtrations $F^{\text{dim}}_i$. 

\subsubsection{}\label{zerodiff} Let $X$ be a quasi-smooth scheme. There exists an open set $U\subset X$ of codimension $\geq 1$ and a vector bundle $\mathcal{E}$ on $U^{\text{cl}}$ such that
\[\mathcal{O}_U\cong \mathcal{O}_{U^{\text{cl}}}[\mathcal{E}[1];d],\]
where $d:\E\to\mathcal{O}_U$ is the zero map. The bundle $\mathcal{E}$ has rank $\dim X^{\text{cl}}-\dim X$.


\subsection{The perverse filtration in cohomology}\label{percoh}
Let $S$ be a scheme over $\mathbb{C}$. 
Let $D^b_c(S)$ be the derived category of bounded complexes of rational constructible sheaves \cite[Section 2]{dCM2}. Consider the perverse $t$-structure $\left(\mathcal{P}^{\leq i}, \mathcal{P}^{\geq i}\right)_{i\in\mathbb{Z}}$ on this category. There are functors: 
\begin{align*}
{}^p\tau^{\leq i}:&D^b_c(S)\to \mathcal{P}^{\leq i},\\ {}^p\tau^{\geq i}:&D^b_c(S)\to \mathcal{P}^{\geq i}
\end{align*}
such that
for $F\in D^b_c(S)$ there is a distinguished triangle in $D^b_c(S)$:
\[{}^p\tau^{\leq i}F\to F\to {}^p\tau^{\geq i+1}F\xrightarrow{[1]}.\]
For a proper map $f:X\to S$ and $F\in D^b_c(X)$, the perverse filtration on $H^\cdot(X,F)$ is defined as the image of
\[{}^pH^{\leq i}_f(X,F):=H^\cdot(S,{}^p\tau^{\leq i}Rf_*F)\to H^\cdot(S, Rf_*F)=H^\cdot(X,F).\]
For $F=IC_X$, the decomposition theorem implies that
\[{}^pIH^{\leq i}_f(X)\hookrightarrow IH^\cdot(X).\]
Let $f:X\to S$ be a generically finite morphism from $X$ smooth, let $U$ be a smooth open subset of $S$ such that $f^{-1}(U)\to U$ is étale, and let $L=f_*\left(\mathbb{Q}_{f^{-1}(U)}\right)$.

For $V\hookrightarrow S$, denote by $X_V:=f^{-1}(V)$. Let $A_V$ be the set of irreducible components of $X_V$. Let $c^a_V$ be the codimension on $\iota^a_V: X_V^a\hookrightarrow X$ for $a\in A_V$. Consider a resolution of singularities $\pi_V^a: Y_V^a\to X_V^a$. 
Let $g_V^a:=f\pi_V^a:Y_V^a\to V$.
Then
\[{}^p\tau^{\leq 0}Rf_*IC_X\cong \text{ker}\left(Rf_*IC_X\to \bigoplus_{V\subsetneqq S}\bigoplus_{a\in A_V}\left({}^p\tau^{>c_V^a}Rg^a_{V*}IC_{Y^a_V}\right)[c^a_V]\right).\]
Define the subspace
\[{}^p\widetilde{\tau}^{\leq 0}Rf_*IC_X:=\text{image}\left(\bigoplus_{V\subsetneqq S}\bigoplus_{a\in A_V}\left({}^p\tau^{\leq-c_V^a}Rg^a_{V*}IC_{Y^a_V}\right)[-c^a_V]\to {}^p\tau^{\leq 0}Rf_*IC_X\right).\]
By a computation of Corti--Hanamura \cite[Proposition 1.5, Theorem 2.4]{CH2}, there is an isomorphism:
\begin{equation}\label{comp1}
IC_S(L)\cong{}^p\tau^{\leq 0}Rf_*IC_X\big/{}^p\widetilde{\tau}^{\leq 0}Rf_*IC_X.
\end{equation}
Further, consider a more general morphism $f:X\to S$ with $X$ smooth. Let $V\subsetneqq S$ be a subvariety. For $i\in\mathbb{Z}$, denote by ${}^p\mathcal{H}^i(Rf_*IC_X)_V$ the direct sum of simple summands of ${}^p\mathcal{H}^i(Rf_*IC_X)$ with support equal to $V$.
A computation of Corti--Hanamura \cite[Proposition 1.5]{CH2} shows that:
\begin{equation}\label{comp2}
{}^p\mathcal{H}^i\left(Rf_*IC_X\right)_V\hookrightarrow \bigoplus_{a\in A_V}{}^{p}\mathcal{H}^{i+c^a_V}\left(Rg^a_{V*}IC_{Y^a_V}\right).
\end{equation}

\section{The perverse filtration in $K$-theory}\label{3}

\subsection{Definition of the filtration $P'^{\leq \cdot}$}\label{Kper}
Let 
$f:X\to S$ be a proper map between varieties. We define an increasing filtration 
\[P'^{\leq i}_f\text{gr}^\cdot G_\cdot(X)\subset \text{gr}^\cdot G_\cdot(X).\]
It induces a filtration on $\text{gr}^\cdot K_\cdot(X)$. 
We use the notations from Subsection \ref{percoh}. 
Let $Y\hookrightarrow X$ be a subvariety and let $T\xrightarrow{\pi}S$ be a map generically finite onto its image from $T$ smooth.
Consider the diagram:
\begin{equation*}
    \begin{tikzcd}
    T\times X\arrow[d,"q"]\arrow[r,"p"]&X\arrow[d,"f"]\\
    T\arrow[r,"\pi"]&S.
    \end{tikzcd}
\end{equation*}
For a correspondence $\Gamma\in \text{gr}_{\dim X-s}G_{T\times_S Y, 0}(T\times X)$, define
\[\Phi_\Gamma:=p_*(\Gamma\otimes q^*(-)):\text{gr}^\cdot K_i(T)\to \text{gr}^{\cdot-s} G_{Y, i}(X).\]
We usually drop the shift by $s$ in the superscript of $\text{gr}^\cdot G_Y(X)$.
We define the subspace of $\text{gr}^\cdot G_Y(X)$:  
\begin{align*}
P'^{\leq i}_{f,T}&:=\text{span}_\Gamma\left(\Phi_\Gamma: \text{gr}^\cdot K_\cdot(T)\to \text{gr}^\cdot G_Y(X)\right)\\
P'^{\leq i}_{f}&:=\text{span}\left(P'^{\leq i}_{f,T}\text{ for all maps }\pi\text{ as above}\right),\end{align*}
where the dimension of the correspondence satisfies 
\begin{equation}\label{ranges}
    \left\lfloor \frac{i+\dim X-\dim T}{2}\right\rfloor\geq s.
\end{equation}
We also define a quotient of $\text{gr}^\cdot G_Y(X)$:
\[P'^{\leq i}_{f}\text{gr}^\cdot G_Y(X)\hookrightarrow \text{gr}^\cdot G_Y(X)\twoheadrightarrow P'^{>i}_{f}\text{gr}^\cdot G_Y(X).\]


\subsection{Functoriality of the filtration $P'^{\leq \cdot}$}

\begin{prop}\label{pullback}
Let $X$ and $Y$ be smooth varieties with $c=\dim X-\dim Y$. 
Consider proper maps
\begin{equation*}
    \begin{tikzcd}[column sep=small]
    Y\arrow[dr,"g"']\arrow[rr,"h"] && X\arrow[dl,"f"]\\
    & S. &
    \end{tikzcd}
\end{equation*}
There are induced maps
\[h^*: P'^{\leq i-c}_{f}\text{gr}^\cdot K_\cdot(X)\to P'^{\leq i}_{g}\text{gr}^\cdot K_\cdot(Y).\]
\end{prop}

\begin{proof}
Let $T\to S$ be a generically finite map onto its image with $T$ smooth. It suffices to show that
\[h^*: P'^{\leq i-c}_{f,T}\KK\to P'^{\leq i}_{g,T}\,\text{gr}^\cdot K_\cdot(Y).\]
Consider the diagram:
\begin{equation*}
    \begin{tikzcd}
    Y\arrow[rr,"h"]& &X\\
    Y\times T\arrow[u,"p_Y"]\arrow[rr,"\widetilde{h}"]\arrow[dr,"q_Y"']& & X\times T\arrow[u,"p_X"]\arrow[dl,"q_X"]\\
     &T&
    \end{tikzcd}
\end{equation*}
Let $\Theta\in \text{gr}_{\dim X-s}G_{T\times_S X, 0}(T\times X)$ be a correspondence such that \[i\geq 2s-\dim X+\dim T.\] For $j\in\mathbb{Z}$, we have that:
\begin{equation*}
    \begin{tikzcd}
    \text{gr}^j K_\cdot(T)\arrow[r,"\Phi_\Theta"]\arrow[dr,"\Phi_{\widetilde{h}^*\Theta}"']&\text{gr}^{j-s} K_\cdot(X)\arrow[d,"h^*"]\\
     &\text{gr}^{j-s} K_\cdot(Y).
    \end{tikzcd}
\end{equation*}
To see this, we compute:
\[h^*\Phi_\Theta(F)=h^*p_{X*}(\Theta\otimes q_X^*F)=p_{Y*}\widetilde{h}^*(\Theta\otimes q_X^*F)=p_{Y*}(\widetilde{h}^*\Theta\otimes q_Y^*F)=\Phi_{\widetilde{h}^*\Theta}(F).\]
The correspondence $\widetilde{h}^*\Theta\in \text{gr}_{\dim Y-s}G_{T\times_S Y}(T\times Y)$ satisfies \[i+c\geq 2s-\dim Y+\dim T,\] and this implies the desired conclusion.
\end{proof}

\begin{prop}\label{pushfor}
Let $X$ and $Y$ be varieties with proper maps
\begin{equation*}
    \begin{tikzcd}[column sep=small]
    Y\arrow[dr,"g"']\arrow[rr,"h"] && X\arrow[dl,"f"]\\
    & S &
    \end{tikzcd}
\end{equation*}
Let $c=\dim X-\dim Y$. 
There are induced maps
\[h_*: P'^{\leq i-c}_{g}\text{gr}_\cdot G_\cdot(Y)\to P'^{\leq i}_{f}\text{gr}_\cdot G_\cdot(X).\]
\end{prop}

\begin{proof}
Let $T\to S$ be a generically finite map onto its image from $T$ smooth. We first explain that
\[h_*: P'^{\leq i-c}_{g,T}\text{gr}_\cdot G_\cdot(Y)\to P'^{\leq i}_{f,T}\,\text{gr}_\cdot G_\cdot(X).\]
We use the notation from the proof of Theorem \ref{pullback}. Consider a correspondence $\Gamma\in \text{gr}_{\dim Y-s}G_{T\times_S Y, 0}(T\times Y)$ such that
\[i\geq 2s-\dim Y+\dim T.\]
For $j\in\mathbb{Z}$, we have that:
\begin{equation*}
    \begin{tikzcd}
    \text{gr}_{\dim T-j} K_\cdot(T)\arrow[r,"\Phi_\Gamma"]\arrow[dr,"\Phi_{\widetilde{h}_*\Gamma}"']&\text{gr}_{\dim Y-j+s} G_\cdot(Y)\arrow[d,"h_*"]\\
     &\text{gr}_{\dim Y-j+s} G_\cdot(X).
    \end{tikzcd}
\end{equation*}
To see this, we compute:
\[h_*p_{Y*}(\Gamma\otimes q_Y^*F)=p_{X*}\widetilde{h}_*(\Gamma\otimes \widetilde{h}^*q_X^*F)=p_{X*}(\widetilde{h}_*\Gamma\otimes q_X^*F).\]
The correspondence \[\widetilde{h}_*\Gamma\in \text{gr}_{\dim Y-s}G_{T\times_S X}(T\times X)=\text{gr}_{\dim X-(c+s)}G_{T\times_S X}(T\times X)\] satisfies
\[i+c\geq 2(s+c)-\dim X+\dim T,\]
and thus the conclusion follows.
\end{proof}

We continue with some further properties of the filtration $P'^{\leq \cdot}$. The following is immediate:

\begin{prop}
Let $f:X\to S$ be a proper map. Let $U$ be an open subset of $S$, $X_U:=f^{-1}(U)$, $\iota: X_U\hookrightarrow X$, and $f_U:X_U\to U$. Then \[\iota^*: P'^{\leq i}_{f}\text{gr}_\cdot G_\cdot(X)\to P'^{\leq i}_{f_U}\text{gr}_\cdot G_\cdot(X_U).\]
\end{prop}


\begin{prop}\label{eulermult}
Let $f:X\to S$ be a proper map from $X$ smooth and consider $e\in \text{gr}^jK_0(X)$. Then 
\[e\cdot P'^{\leq i}_f\text{gr}^a K_\cdot(X)\subset P'^{\leq i+2j}_f\text{gr}^{a+j} K_\cdot(X).\]
\end{prop}

\begin{proof}
Let $T\to S$ be a generically finite map onto its image with $T$ smooth and let $\Theta\in \text{gr}_aG_{T\times_SX, 0}(T\times X)$. Let $p:T\times X\to X$ be the natural projection. Then
\[p^*(e)\cdot \Theta\in \text{gr}_{a-j}G_{T\times_SX, 0}(T\times X).\]
For $x\in \text{gr}_\cdot K_\cdot(T)$, we have that
\[e\cdot \Phi_\Theta(x)=\Phi_{p^*(e)\cdot \Theta}(x),\] and the conclusion thus follows.
\end{proof}

\begin{prop}\label{smpr}
Let $X$ and $Y$ be smooth varieties with proper maps
\begin{equation*}
    \begin{tikzcd}[column sep=small]
    Y\arrow[dr,"g"']\arrow[rr,"h"] && X\arrow[dl,"f"]\\
    & S &
    \end{tikzcd}
\end{equation*}
such that $h$ is surjective. Let $c=\dim X-\dim Y$. Then \begin{align*}
    h_*\left(P'^{\leq i}_f\text{gr}_\cdot K_\cdot(Y)_\mathbb{Q}\right)&=P'^{\leq i+c}_f\text{gr}_\cdot K_\cdot(X)_\mathbb{Q}\\
    h^*\big(\text{gr}^\cdot K(X)_\mathbb{Q}\big)\cap P'^{\leq i+c}_g\text{gr}^\cdot K_\cdot(Y)_\mathbb{Q}&=h^*\left(P'^{\leq i}_f\text{gr}^\cdot K_\cdot(X)_\mathbb{Q}\right).
    \end{align*}
    If there exists $X'\to Y$ such that the induced map $X'\to X$ is birational, then the above isomorphisms hold integrally.
\end{prop}

\begin{proof}
The statement and its proof are similar to \cite[Proposition 3.11]{CH2}.

Let $i:X'\to Y$ be a map such that $hi:X'\to X$ is generically finite and surjective. Let $f':=fi:X'\to S$. Then, by Proposition \ref{pushfor}:
\[P'^{\leq i+c}_{f'}\text{gr}_\cdot K_\cdot(X')\xrightarrow{i_*} P'^{\leq i}_g\text{gr}_\cdot K_\cdot(Y)\xrightarrow{h_*} P'^{\leq i+c}_f\text{gr}_\cdot K_\cdot(X).\]
The map $h_*i_*:\text{gr}_\cdot K_\cdot(X')\to \text{gr}_\cdot K_\cdot(X)$ is multiplication by the degree of the map $hi$, so it is an isomorphism rationally. It is an isomorphism integrally if $X'\to X$ has degree $1$. The pullback statement is similar.
\end{proof}

\subsection{The filtration $P^{\leq \cdot}$}\label{filtPP}

Let $f:X\to S$ be a proper map from $X$ smooth. Let $V\hookrightarrow S$ be a subvariety, and let $A_V$ the set of irreducible components of $f^{-1}(V)$. For an irreducible component $X^a_V$ of $f^{-1}(V)$, consider a resolution of singularities $\pi^a_V$ as follows:
\begin{equation*}
    \begin{tikzcd}
    \widetilde{X^a_V}\arrow[dr, "\widetilde{f^a_V}"']\arrow[r,"\pi^a_V"]& X^a_V\arrow[r, hook, "\iota^a_V"]\arrow[d,"f^a_V"]&X\arrow[d,"f"]\\
     &V\arrow[r, hook]&S.
    \end{tikzcd}
\end{equation*}
Let $c^a_V$ be the codimension of $X^a_V$ in $X$. Denote by $\tau^a_V=\iota^a_V\pi^a_V$.
Consider a subvariety $Y\hookrightarrow X$. Define
\[P^{\leq i}_f\text{gr}^\cdot G_Y(X):=\bigcap_{V\subsetneqq S}\bigcap_{a\in A_V}\text{ker}\left(\tau^{a*}_V:P'^{\leq i}_f\text{gr}^\cdot G_Y(X)\to P'^{>i+c^a_V}_{\widetilde{f^a_V}}\text{gr}^\cdot K_\cdot\left(\widetilde{X^a_V}\right)\right).\]
The definition is independent of the resolutions $\pi^a_V$ chosen. Indeed, consider two different resolutions $\widetilde{X^a_V}$, $\widetilde{X'^a_V}$. There exists $W$ such that
\begin{equation*}
    \begin{tikzcd}
     & W\arrow[dl,"\pi"']\arrow[dr,"\pi'"]& \\
     \widetilde{X^a_V}\arrow[dr]& &\widetilde{X'^a_V}\arrow[dl]\\
      &X^a_V&
    \end{tikzcd}
\end{equation*}
where the maps $\pi$ and $\pi'$ are successive blow-ups along smooth subvarieties of $\widetilde{X^a_V}$ and $\widetilde{X'^a_V}$, respectively. 
Let $\tau'^a_V: \widetilde{X'^a_V}\to X$ as above. Then $\tau^a_V\pi=\tau'^a_V\pi'$.
By Proposition \ref{smpr},
\begin{multline*}
    \text{ker}\left(\tau^{a*}_V:P'^{\leq i}_f\text{gr}^\cdot G_Y(X)\to P'^{>i+c^a_V}_{\widetilde{f^a_V}}\text{gr}^\cdot K_\cdot\left(\widetilde{X^a_V}\right)\right)\cong\\
    \text{ker}\left(\pi^*\tau^{a*}_V:P'^{\leq i}_f\text{gr}^\cdot G_Y(X)\to P'^{>i+c^a_V}_{\widetilde{f^a_V}}\text{gr}^\cdot K_\cdot(W)\right)\cong\\
    \text{ker}\left(\tau'^{a*}_V:P'^{\leq i}_f\text{gr}^\cdot G_Y(X)\to P'^{>i+c^a_V}_{\widetilde{f^a_V}}\text{gr}^\cdot K_\cdot\left(\widetilde{X'^a_V}\right)\right).
\end{multline*}

\begin{thm}\label{functorialityP}
Let $X$ and $Y$ be smooth varieties with $c=\dim X-\dim Y$. 
Consider proper maps
\begin{equation*}
    \begin{tikzcd}[column sep=small]
    Y\arrow[dr,"g"']\arrow[rr,"h"] && X\arrow[dl,"f"]\\
    & S. &
    \end{tikzcd}
\end{equation*}
There are induced maps
\begin{align*}
h^*&: P^{\leq i-c}_{f}\text{gr}^\cdot K_\cdot(X)\to P^{\leq i}_{g}\text{gr}^\cdot K_\cdot(Y)\\
h_*&: P^{\leq i-c}_{g}\text{gr}_\cdot K_\cdot(Y)\to P^{\leq i}_{f}\text{gr}_\cdot K_\cdot(X).
\end{align*}
\end{thm}

\begin{proof}
The functoriality of $h^*$ follows from Proposition \ref{pullback} and induction on dimension of $S$.

We discuss the statement for $h_*$. We use induction on the dimension of $S$. The case of $S$ a point is clear as $P'^{\leq i}_f=P^{\leq i}_f$. We use the notation from the beginning of Subsection \ref{filtPP}. Let $V$ be a subvariety of $S$. Let $X^a_V$ be an irreducible component of $f^{-1}(V)$ with a resolution of singularities $\widetilde{X^a_V}\to X^a_V$. Let $B$ be the set of irreducible component of $Y_V$ over $X_V^a$. For $b\in B$, consider a resolution of singularities $\widetilde{Y^b_V}\to Y^b_V$ and maps such that
\begin{equation*}
    \begin{tikzcd}
    \bigsqcup_{b\in B}\widetilde{Y^b_V}\arrow[r,"\bigoplus_B h^b_V"]\arrow[d,"\bigoplus_B \tau^b_V"]& \widetilde{X^a_V}\arrow[d,"\tau^a_V"]\\
    Y\arrow[r,"h"]&X.
    \end{tikzcd}
\end{equation*}
Consider the cartesian diagram
\begin{equation*}\label{yder}
    \begin{tikzcd}
    Y^{\text{der}}_V\arrow[d,"\tau"]\arrow[r,"\widetilde{h}"]&\widetilde{X^a_V}\arrow[d,"\tau^a_V"]\\
    Y\arrow[r,"h"]&X.
    \end{tikzcd}
\end{equation*}
The scheme $Y^{\text{der}}_V$ is quasi-smooth, see Subsection \eqref{qss}, and $\text{reldim}\,\widetilde{h}=\text{reldim}\,h$. 
For $b\in B$, there is a map $p_b:\widetilde{Y^b_V}\to Y^{\text{der}}_V.$ Let $d_b=\dim \widetilde{Y^b_V}-\dim Y^{\text{der}}_V$ and define 
\[e_b=\text{det}\left(\mathbb{L}_{\tau^b_V}/h^{b*}_V\mathbb{L}_{\tau^a_V}\right)\in \text{gr}^{d_b}K_0\left(\widetilde{Y^b_V}\right).\] 
By a version of the excess intersection formula, the following diagram commutes:
\begin{equation}\label{eul}
    \begin{tikzcd}
    \text{gr}_\cdot K_\cdot(Y)\arrow[r,"h_*"]\arrow[d,"\bigoplus_{B} \tau^{b*}_V"]
    &\text{gr}_\cdot K_\cdot(X)\arrow[dd, "\tau^{a*}_V"]\\
    \bigoplus_B \text{gr}_\cdot K_\cdot(\widetilde{Y^b_V})\arrow[d,"\bigoplus_{ B}e_b"]& \\
    \bigoplus_B \text{gr}_\cdot K_\cdot(\widetilde{Y^b_V})\arrow[r,"\bigoplus_B h^b_{V*}"]& K_\cdot(\widetilde{X^a_V}),
    \end{tikzcd}
\end{equation}
where we have ignored shifts in the above gradings. We now explain that \eqref{eul} commutes.
Consider the diagram
\begin{equation*}
    \begin{tikzcd}
    \bigsqcup_B \widetilde{Y^b_V} \arrow[drr, bend left, "\bigoplus_B h^b_V"]\arrow[ddr, bend right, "\bigoplus_{B} \tau^{b}_V"']\arrow[dr, "\sqcup_Bp_b"]& & \\
     &Y^{\text{der}}_V\arrow[d, "\tau"]\arrow[r,"\widetilde{h}"]& \widetilde{X^a_V}\arrow[d,"\tau^a_V"]\\
      &Y\arrow[r,"h"]&X.
    \end{tikzcd}
\end{equation*}
Then 
\[\sum_{b\in B} h^b_{V*}\left(e_b\cdot \tau^{b*}_V\right)=\sum_{b\in B}\widetilde{h}_*p_{b*}\left(e_b\cdot p_b^*\tau^*\right)=\widetilde{h}_*\left(\left(\sum_{b\in B}p_{b*} e_b\right)\cdot \tau^*\right).\]
For $M$ a quasi-smooth scheme, denote by $1:=[\mathcal{O}_M]\in \text{gr}^0K_0(M).$
It suffices to show that 
\begin{equation}\label{sumone}
\sum_{b\in B}p_{b*}(e_b)=1\in \text{gr}^0K_0\left(Y^{\text{der}}_V\right).
\end{equation}

The underlying scheme $Y^{\text{cl}}_V$ has irreducible components indexed by $B$ and these componenets are birational to $\widetilde{Y^b_V}$. 
Recall the discussion in Subsection \ref{zerodiff}. There exist open sets 
\begin{align*}
    W=\bigsqcup_{b\in B}W^b&\subset Y^{\text{der}}_V,\\
    U^b&\subset \widetilde{Y^b_V}
\end{align*}
whose complements have codimension $\geq 1$ and such that for any $b\in B$:
\begin{align*}
W^{b,\text{cl}}&\cong U^b,\\
W^b\times_{Y^{\text{der}}_V}\widetilde{Y^b_V}&\cong 
    U^b,\\
    \mathcal{O}_{W^b}&\cong\mathcal{O}_{U^b}\left[\mathcal{E}^b[1];d\right],
\end{align*}
where $\mathcal{E}^b$ is a vector bundle on $U^b$ of rank $d_b$ and the differential $\mathcal{E}^b\to\mathcal{O}_{U^b}$ is zero. 
Let $\varepsilon_b\in \text{gr}^{d_b}K_0\left(U^b\right)$ be the Euler class of $\mathcal{E}^b$. Then $p_{b*}\left(\varepsilon_b\right)=1\in \text{gr}^0K_0(W^b)$ and the restriction map sends
\begin{align*}
    \text{res}: \text{gr}^{d_b}K_0\left(\widetilde{Y^b_V}\right)&\to \text{gr}^{d_b}K_0\left(U^b\right)\\
    e_b&\mapsto \varepsilon_b.
\end{align*}
Back to proving \eqref{sumone}, 
we have that $\text{gr}^0K_0\left(Y^{\text{der}}_V\right)\cong \bigoplus_{b\in B}\text{gr}^0K_0\left(W^b\right)$. Consider the diagram
\begin{equation*}
    \begin{tikzcd}
    \text{gr}^{d_b}K_0\left(\widetilde{Y^b_V}\right)\arrow[d,"p_{b*}"]\arrow[r,"\text{res}"]&\text{gr}^{d_b}K_0\left(U^b\right)\arrow[d,"p_{b*}"]\\
    \text{gr}^{0}K_0\left(Y^{\text{der}}_V\right)\arrow[r,"\text{res}"]& \text{gr}^0K_0(W^b),
    \end{tikzcd}
\end{equation*}
where the horizontal maps are restriction to open sets maps. Then
\[\text{res}\,p_{b*}(e_b)=
p_{b*}\left(\varepsilon_b\right)=1
\text{ in }\text{gr}^0K_0(W^b).\]
The diagram \eqref{eul} thus commutes.
The conclusion now follows from Propositions \ref{pushfor} and \ref{eulermult}. 
\end{proof}

\subsection{Towards the filtration $\P^{\leq i}_f$}\label{towa}

We continue with the notation from Subsection \ref{Kper}. Let $X$ be a smooth variety with a proper map $f:X\to S$. Let $T\xrightarrow{\pi}S$ be a generically finite map onto its image from $T$ smooth. 

We say that $\Gamma$ is a \textit{$(f, \pi)$-quasi-smooth scheme} if $\Gamma$ is a derived scheme with maps
\begin{equation*}
    \begin{tikzcd}
    &X'\arrow[d,"t"]\\
    \Gamma\arrow[d,"q"]\arrow[ur,hook, "\iota"]\arrow[r,"p"]&X\arrow[d,"f"]\\
    T\arrow[r, "\pi"]&S
    \end{tikzcd}
\end{equation*} such that $\iota$ is a quasi-smooth immersion in a smooth variety $X'$ (i.e. the cotangent complex $\mathbb{L}_\iota[-1]$ is a vector bundle on $\Gamma$), $t$ is smooth, and $q^{\text{cl}}$ is surjective. 
The conditions on the maps $\iota$ and $t$ imply that $\Gamma$ is quasi-smooth.
Let \[\text{gr}^\cdot K^q_{T\times_SX}(T\times X)\subset\text{gr}^\cdot K_{T\times_SX}(T\times X)\] be the subspace generated by classes $[\Gamma]$ for $(f, \pi)$-quasi-smooth schemes as above. 

\begin{prop}\label{quasi}
Let $h$ be a proper map:
\begin{equation*}
    \begin{tikzcd}[column sep=small]
    Y\arrow[dr,"g"']\arrow[rr,"h"] && X\arrow[dl,"f"]\\
    & S. &
    \end{tikzcd}
\end{equation*}
There are induced maps 
\[h_*:\text{gr}_\cdot K^q_{T\times_SY}(T\times Y)\to \text{gr}_\cdot K^q_{T\times_SX}(T\times X).\]
If $h$ is surjective, then there are induced maps
\[h^*: \text{gr}^\cdot K^q_{T\times_SX}(T\times X)\to \text{gr}^\cdot K^q_{T\times_SY}(T\times Y).\]
\end{prop}

\begin{proof}
We discuss the statement about pullback. Consider the diagram:
\begin{equation*}
    \begin{tikzcd}
    \Theta\arrow[d, "r"]\arrow[r, hook]&Y'\arrow[d,"h'"]\arrow[r,"t_Y"]&Y\arrow[d,"h"]\\
    \Gamma\arrow[d,"q"]\arrow[r, hook]&X'\arrow[r,"t_X"]&X\arrow[dl,"f"]\\
    T\arrow[r, "\pi"]&S,
    \end{tikzcd}
\end{equation*}
where $\Gamma$ is a $(f, \pi)$-quasi-smooth scheme with $q^{\text{cl}}$ is surjective, $t_X$ is smooth, and the upper squares are cartesian. Then
the map $\Theta\hookrightarrow Y'$ is a quasi-smooth immersion and $t_Y$ is smooth. 
The map $h$ is surjective, so $r^{\text{cl}}:\Theta^{\text{cl}}\to \Gamma^{\text{cl}}$ is surjective, and thus $(qr)^{\text{cl}}: \Theta^{\text{cl}}\to T$ is surjective as well, so $\Theta$ is a $(g,\pi)$-quasi-smooth scheme.

We next discuss the statement about pushforward. Consider 
\begin{equation*}
    \begin{tikzcd}
    &Y'\arrow[d,"t"]\\
    \Gamma\arrow[d,"q"]\arrow[ur,hook, "\iota"]\arrow[r,"p"]&Y\arrow[d,"g"]\\
    T\arrow[r]&S
    \end{tikzcd}
\end{equation*} such that $\iota$ is a closed immersion, $t$ is smooth, and $q^{\text{cl}}$ is surjective. The map $Y'\to X$ is a proper map of smooth quasi-projective varieties, so we can choose $X'$ with maps
\[Y'\xrightarrow{\iota'} X'\xrightarrow{t'}X\] such that $\iota'$ is a closed immersion and $t'$ is smooth. Then
\begin{equation*}
    \begin{tikzcd}
    &X'\arrow[d,"t'"]\\
    \Gamma\arrow[d,"q"]\arrow[ur,hook, "\iota'\iota"]\arrow[r,"hp"]&X\arrow[d,"f"]\\
    T\arrow[r]&S
    \end{tikzcd}
\end{equation*}
such that $\iota'\iota$ is a closed immersion, $t'$ is smooth, and $q^{\text{cl}}$ is surjective. 
\end{proof}

Consider a diagram 
\begin{equation}\label{prime}
    \begin{tikzcd}
    & X'\arrow[d,"t"]\\
    \Gamma\arrow[d,"q"]\arrow[r,"p"]\arrow[ru,hook,"\iota"]&X\arrow[d,"f"]\\
    T\arrow[r,"\pi"]&S
    \end{tikzcd}
\end{equation} 
as above, with $t$ a smooth map and with $\iota$ a closed immersion. Let \[T\times_SX=Z_1\cup Z_2,\] where $Z_1$ is the union of irreducible components of $T\times_SX$ dominant over $T$ and $Z_2$ is the union of the other irreducible components. Denote by $Z_1^o:=Z_1-\left(Z_1\cap Z_2\right)$. Similarly define $Z'_1$ and $Z'_2$ for $T\times_SX'$. 
Let $b=\text{reldim}\,q$ and $a=b+\dim T=\dim \Gamma$. 

\begin{prop}\label{reszero}
The class $[\Gamma]\in \text{gr}_aK_{T\times_SX'}(T\times X')$ is not supported on $Z'_2$.
\end{prop}

\begin{proof}
Let $\ell$ be an $ft$-ample divisor. Denote by $\text{pr}_1:T\times X'\to T$.
Then
\begin{equation}\label{inte}
\text{pr}_{1*}\left([\Gamma]\cdot \ell^b\right)=d[T]\in \text{gr}_{\dim T}K_\cdot(T)
\end{equation}
for $d$ a non-zero integer. Indeed, let $\eta$ be the generic point of $T$. By abuse of notation, we denote by $\eta$ its image in $S$.
It suffices to show \eqref{inte} after restricting to $\eta$. In this case, $d$ is the intersection number $\left[\Gamma_\eta\right]\cdot\ell^b$ in $X'_{\eta}$.

Further, let $x\in \text{gr}_aK_{Z'_2}(T\times X')$. We have that \[\text{pr}_{1*}\left(x\cdot \ell^b\right)=0\in \text{gr}_{\dim T} K_\cdot(T)\] because the support on $x\cdot \ell^b$ is not dominant over $T$. The conclusion thus follows.
\end{proof}

\begin{prop}\label{lb}
Let $T\xrightarrow{\pi}X$ be a generically finite map from $T$ smooth with image $V$. 
Let $a>\dim X_V$. Then $\text{gr}_aK^q_{T\times_SX}(T\times X)=0$. Further, $\text{gr}_{\dim X_V}K^q_{T\times_SX}(T\times X)$ is generated by irreducible components of $T\times_SX$ dominant over $T$ of dimension $X_V$.
\end{prop}

\begin{proof}
Suppose we are in the setting of \eqref{prime} and let $s:X\to X'$ be a section of $t$. We write $\widetilde{p}:\Gamma\to T\times X$, $\widetilde{\iota}:\Gamma\to T\times X'$ etc.
Assume that \[\widetilde{t}_*\widetilde{\iota}_*[\Gamma]=\widetilde{p}_*[\Gamma]\neq 0\in \text{gr}_aK^q_{T\times_SX}(T\times X).\] Then there exists a non-zero $x\in \text{gr}_aK^q_{T\times_SX}(T\times X)$ such that \[\widetilde{p}_*[\Gamma]=\widetilde{s}_*(x)\in \text{gr}_aK^q_{T\times_SX'}(T\times X').\] 
Consider the diagram
\begin{equation*}
    \begin{tikzcd}
    \text{gr}_aK_{T\times_SX'}(T\times X')\arrow[r,"\text{res}"]& \text{gr}_aK_{Z'^o_1}(T\times X'\setminus Z'_2)\\
    \text{gr}_aK_{T\times_SX}(T\times X)\arrow[r,"\text{res}"]\arrow[u,"\widetilde{s}_*"]& \text{gr}_aK_{Z_1^o}(T\times X\setminus Z_2).\arrow[u,"\widetilde{s}_*"]
    \end{tikzcd}
\end{equation*}
By Proposition \ref{reszero}, we have that $\text{res}(x)\neq 0\in \text{gr}_aK_{Z_1^o}(T\times X\setminus Z_2)$. We have that $\dim Z_1^o=\dim X_V$, and the conclusion follows from here.
\end{proof}

\subsection{The perverse filtration $\P^{\leq i}_f$}
\label{Kper2}
We now define a smaller filtration $\P^{\leq i}_f\subset P^{\leq i}_f$.
We use the notation from Subsection \ref{Kper}. 

Let $X$ be a smooth variety with a proper map $f:X\to S$ and let $T\xrightarrow{\pi}S$ be a generically finite map onto its image from $T$ smooth. Consider a subvariety $Y\hookrightarrow X$. 
Define the subspaces of $\text{gr}^\cdot G_Y(X)$:  
\begin{align*}
\P'^{\leq i}_{f,T}&:=\text{span}_\Gamma\left(\Phi_\Gamma: \text{gr}^\cdot K_\cdot(T)\to \text{gr}^\cdot G_Y(X)\right)\\
\P'^{\leq i}_{f, V}&:=\text{span}\left(\P'^{\leq i}_{f,T}\text{ for all maps }\pi\text{ as above }V\right),\end{align*}
where $\Gamma\in \text{gr}_{\dim X-s}K^q_{T\times_S Y, 0}(T\times X)$
and 
\begin{equation*}
    \left\lfloor \frac{i+\dim X-\dim T}{2}\right\rfloor\geq s.
\end{equation*}
Using the notation from Subsection \ref{filtPP}, define 
\[\P^{\leq i}_f\text{gr}^\cdot G_Y(X):=\bigcap_{V\subsetneqq S}\bigcap_{a\in A_V}\text{ker}\left(\tau^{a*}_V:\P'^{\leq i}_f\text{gr}^\cdot G_Y(X)\to P'^{> i+c^a_V}_{\widetilde{f^a_V}}\text{gr}^\cdot K_\cdot\left(\widetilde{X^a_V}\right)\right).\]
The definition is independent of the resolutions $\widetilde{X^a_V}$ chosen, see Subsection \ref{filtPP}.

\begin{thm}\label{functorP}
Let $X$ and $Y$ be smooth varieties with $c=\dim X-\dim Y$. 
Consider proper maps
\begin{equation*}
    \begin{tikzcd}[column sep=small]
    Y\arrow[dr,"g"']\arrow[rr,"h"] && X\arrow[dl,"f"]\\
    & S. &
    \end{tikzcd}
\end{equation*}
There are induced maps
\begin{align*}
h_*&: \P'^{\leq i-c}_{g}\text{gr}_\cdot K_\cdot(Y)\to \P'^{\leq i}_{f}\text{gr}_\cdot K_\cdot(X)\\
h_*&: \P^{\leq i-c}_{g}\text{gr}_\cdot K_\cdot(Y)\to \P^{\leq i}_{f}\text{gr}_\cdot K_\cdot(X).
\end{align*}
If $h$ is surjective, then there are induced maps
\begin{align*}
h^*&: \P'^{\leq i-c}_{f}\text{gr}^\cdot K_\cdot(X)\to \P'^{\leq i}_{g}\text{gr}^\cdot K_\cdot(Y)\\
h^*&: \P^{\leq i-c}_{f}\text{gr}^\cdot K_\cdot(X)\to \P^{\leq i}_{g}\text{gr}^\cdot K_\cdot(Y).
\end{align*}
\end{thm}

\begin{proof}
The functoriality follow as in Propositions \ref{pullback}, \ref{pushfor}, and Theorem \ref{functorialityP}, using Proposition \ref{quasi}.
\end{proof}

\subsection{Properties of the perverse filtration}
Consider a proper map $f:X\to S$ with $X$ smooth. Define the defect of semismallness of $f$ by 
\[s:=s(f)=\dim X\times_S X-\dim X.\]
Further, define $s'=\text{max}\left(\dim X+\dim S-4, \dim X\right)$.
The perverse filtration in cohomology satisfies
\[{}^pH^{\leq -s-1}_f(X)=0\text{ and }{}^pH^{\leq s}_f(X)=H^\cdot(X),\]
see \cite[Section 1.6]{dCM2}.
We prove an analogous result in $K$-theory:

\begin{thm}\label{sd}
For $f$ as above,
\begin{align*}
&P^{\leq -s'-1}_f\text{gr}^\cdot K_\cdot(X)=\P^{\leq -s-1}_f\text{gr}^\cdot K_\cdot(X)=0\\
&P^{\leq s}_f\text{gr}^\cdot K_0(X)=\P^{\leq s}_f\text{gr}^\cdot K_0(X)=\text{gr}^\cdot K_0(X).\end{align*}
\end{thm}


\begin{prop}\label{ppp}
Let $f:X\to S$ be a surjective map from $X$ smooth with $\text{reldim}\,f>0$
and consider a subvariety $Z\hookrightarrow X$ of codimension $\geq 2$. Then there exists a subvariety $\iota: Y\hookrightarrow X$ of codimension $1$ such that $Z\subset Y$ and $f\iota:Y\to S$ is surjective.
\end{prop}

\begin{proof}
It suffices to pass to the generic point of $Z$, and we can thus assume that $Z$ is a point and is given by a complete intersection of smooth hypersurfaces $H_1,\cdots, H_r$ in $X$ with $r\geq 2$. Localizing at the generic point of $Z$, we can assume that $Z$ is a point. Further restricting to an open set of $X$, we can assume that the fibers of $f$ are irreducible. Assume that none of the maps \[f_i: H_i\to Z\] are surjective. Let $S_i$ be the image of $f_i$. 
Let $S':=\bigcap_{i=1}^rS_i$. Then $S'$ is not empty because it contains $f(Z)$. We have $\pi^{-1}(S_i)=H_i$ and so $\bigcap_{i=1}^r H_i$ contains $\pi^{-1}\left(S'\right)$. This means that $\dim \big(\bigcap_{i=1}^r H_i\big)\geq \text{reldim}\,f$. This bound contradicts that $\bigcap_{i=1}^r H_i$ is a point $Z$.
\end{proof}

\begin{prop}\label{reldim}
Let $f:X\to S$ be a proper surjective map from $X$ smooth of relative dimension $d$. 
Then
\[\P'^{\leq d}_{f}\text{gr}^\cdot K_0(X)=\text{gr}^\cdot K_0(X).\]
\end{prop}

\begin{proof} 
We use induction on $d$. 
Assume that $f$ is generically finite. Consider the correspodence $\Delta\cong X\hookrightarrow X\times_SX$:
\begin{equation*}
    \begin{tikzcd}
    \Delta\arrow["\sim",r] \arrow["\sim",d]&X\arrow[d,"f"]\\
    X\arrow[r,"f"]&S.
    \end{tikzcd}
\end{equation*}
This implies that $\P'^{\leq 0}_{f}\text{gr}^\cdot K_\cdot(X)=\text{gr}^\cdot K_\cdot(X)$.

Consider $f$ with $d>0$. Let $\iota: Z\hookrightarrow X$ be a subvariety of codimension $\geq 2$. By Proposition \ref{ppp}, there exists $Y\hookrightarrow X$ of codimension $1$ such that $Z\subset Y$ and $Y\to S$ has image $W$ of codimension $\leq 1$ in $S$. Let $Y'\to Y$ be a resolution of singularities and denote the resulting map by $g:Y'\to W$. 
By induction,
\[\P'^{\leq d-1}_{g}\text{gr}^\cdot K_0(Y')=\text{gr}^\cdot K_0(Y').\]
By Proposition \ref{prop2}, \[\text{image}\,\left(\iota_*:\text{gr}_\cdot G_0(Z)\to \text{gr}_\cdot K_0(X)\right)\subset \text{image}\,\left(g_*:\text{gr}_\cdot K_0(Y')\to \text{gr}_\cdot K_0(X)\right).\]
Finally, assume that $Z\hookrightarrow Y$ has codimension $1$. By Proposition \ref{prop1}, it suffices to show that \[\text{image}\left(\text{gr}_{\dim Z}G_0(Z)\to \text{gr}_{\dim Z}G_0(X)\right)\subset \P'^{\leq d}_f\text{gr}^\cdot K_0(X)\] because $\text{gr}_iG_0(Z)$ for $i<\dim Z$ is generated by varieties of smaller dimension than $Z$.
If $Z\to S$ is surjective, then it has relative dimension $d-1$ and we can treat it as above. If $Z\to S$ is not surjective, let $W\subset S$ be its image. Choose a resolution of singularities $T\to W$ and a smooth variety $\Gamma$ with surjective maps $p$ and $q$:
\begin{equation*}
    \begin{tikzcd}
    \Gamma\arrow[d,"q"]\arrow[r,"p"]&Z\arrow[d]\arrow[r,hook]&X\arrow[d,"f"]\\
    T\arrow[r]&W\arrow[r, hook]&S.
    \end{tikzcd}
\end{equation*}
Then $[\Gamma]\in \text{gr}_{\dim X-1}K^q_{T\times_SX}(T\times X)$ and its image $\Phi_\Gamma$ is in $\P'^{\leq d}_{f}\text{gr}^\cdot K_0(X)$. Then \[\text{image}\left(\text{gr}_{\dim Z}G_0(Z)\to \text{gr}_{\dim Z} K_0(X)\right)\subset \text{image}\,\Phi_\Gamma\subset \P'^{\leq d}_{f}\text{gr}^\cdot K_0(X).\]
The conclusion now follows from Proposition \ref{prop1}.

\end{proof}

\begin{proof}[Proof of Theorem \ref{sd}]
We first show that $P^{\leq -s'-1}_f\text{gr}^\cdot K_\cdot(X)=0$. 
Consider a map $\pi:T\to X$ generically finite onto its image $V\subset S$ with $T$ smooth and consider a correspondence 
\[\Gamma\in \text{gr}_{\dim X-b}G_{T\times_S X}(T\times S).\]
Then $\dim X-b\leq \dim T\times_SX\leq \text{max}\left(\dim X, \dim X+\dim T-2\right)$, and so
\[b\geq \text{min}\left(0, -\dim T+2\right).\]
By the bound \eqref{ranges}, it suffices to show that
\[\left\lfloor \frac{-s'-1+\dim X-\dim T}{2}\right\rfloor< \text{min}\left(0, -\dim T+2\right),\] or, alternatively, that
\[\text{max}\big(\dim X-\dim T-1, \dim X+\dim T-5\big)< s',\]
which is true because $0\leq \dim T\leq \dim S$.
\smallskip

We next explain that $\P^{\leq -s-1}_f\KK=0$. We keep the notation from the previous paragraph. Let $[\Gamma]\in \text{gr}_{\dim X-b}K^q_{T\times_S X}(T\times S).$ By Proposition \ref{lb}, we have that \[b\geq \dim X-\dim X_V.\] It suffices to show that
\[\left\lfloor \frac{-s-1+\dim X-\dim T}{2}\right\rfloor< \dim X-\dim X_V,\] or, alternatively, that
\[2\dim X_V-\dim V\leq s-\dim X=\dim X\times_SX,\]
which is true because $2\dim X_V-\dim V\leq \dim X_V\times_VX_V\leq \dim X\times_SX$.
\smallskip

We next show that $\P^{\leq s}_f\text{gr}^\cdot K_0(X)=\text{gr}^\cdot K_0(X)$. We can assume that $f$ is surjective of relative dimension $d$. Use the notation from Subsection \ref{filtPP}.
We have that
\[\P^{\leq s}_f\text{gr}^\cdot K_0(X):=\bigcap_{V\subsetneqq S}\bigcap_{a\in A_V}\text{ker}\left(\tau^{a*}_V:\P'^{\leq s}_f\text{gr}^\cdot K_0(X)\to P'^{> s+c^a_V}_{\widetilde{f^a_V}}\text{gr}^\cdot K_0\left(\widetilde{X^a_V}\right)\right).\]
We claim that
\[\text{reldim}\,\left(\widetilde{X^a_V}\to V\right)=\text{reldim}\,\left(X^a_V\to V\right)\leq s+c^a_V.\]
Indeed,
\begin{align*}
    \dim X^a_V-\dim V&\leq \left(\dim X\times_SX-\dim X\right)+\left(\dim X-\dim X^a_V\right)\\
    2\dim X^a_V-\dim V&\leq \dim X^a_V\times_VX^a_V\leq  \dim X\times_SX,
\end{align*}
which is true.
By Proposition \ref{reldim}, this implies that $P'^{>s+c^a_V}_{\widetilde{f^a_V}}\text{gr}^\cdot K_0\left(\widetilde{X^a_V}\right)=0$. Furthermore, $s\geq d$, so Proposition \ref{reldim} implies that $\P'^{\leq s}_f\text{gr}^\cdot K_0(X)=\text{gr}^\cdot K_0(X)$, and thus 
$\P^{\leq s}_f\text{gr}^\cdot K_0(X)=\text{gr}^\cdot K_0(X).$ This also implies that $P^{\leq s}_f\text{gr}^\cdot K_0(X)=\text{gr}^\cdot K_0(X)$. 
\end{proof}


\subsection{Examples of perverse filtration in $K$-theory}
\label{examplesfilt}

\subsubsection{}\label{oneex} Let $X$ be a smooth variety of dimension $d$, and let $f:X\to \text{Spec}\,\mathbb{C}.$
Then 
\[P^{\leq i}_f\text{gr}^j K_\cdot(X)=
\begin{cases} \mbox{$\text{gr}^jK_\cdot(X)$} & \mbox{if } j\leq \lfloor \frac{i+d}{2}\rfloor, \\ \mbox{$0$} & \mbox{otherwise.} \end{cases}\]

\subsubsection{}\label{projective} Let $X$ be a smooth variety and let $E$ be a vector bundle on $X$ of rank $d+1$. Let $Y:=\mathbb{P}_X(E)$. Denote by $\hbar:=c_1(\mathcal{O}_Y(1))\in \text{gr}^2 K_0(Y).$ 
Consider the projection map $f:Y\to X$. We have that $s(f)=d$. For $i\leq d$, there exists an isomorphism
\begin{align*}
    \bigoplus_{0\leq j\leq\lfloor \frac{i+d}{2}\rfloor}\text{gr}^{a-2j} K_\cdot(X) &\cong P'^{\leq i}_f\text{gr}^a K_\cdot(Y)\\
    \big(x_0,\cdots, x_{\lfloor \frac{i+d}{2}\rfloor}\big)&\mapsto\sum_{j\leq \lfloor \frac{i+d}{2}\rfloor} \hbar^jf^*(x_j).
\end{align*}
The condition for $P^{\leq i}$ is checked using projective bundles over varieties of smaller dimension, and 
we obtain that
\[\bigoplus_{0\leq j\leq\lfloor \frac{i+d}{2}\rfloor}\text{gr}^{a-2j} K_\cdot(X) \cong P^{\leq i}_f\text{gr}^a K_\cdot(Y).\] 

\subsubsection{}
Let $X$ be a smooth variety and let $Z$ be a smooth subvariety of codimension $d+1$. Consider the blow-up diagram for $Y=\text{Bl}_ZX$:
\begin{equation*}
    \begin{tikzcd}
    E\arrow[r, hook, "\iota"]\arrow[d,"p"]& Y\arrow[d,"f"]\\
    Z\arrow[r, hook, "j"]& X.
    \end{tikzcd}
\end{equation*}
Let $\hbar:=c_1(\mathcal{O}_E(1))\in \text{gr}^2K_0(E)$. We have that $s(f)=d-1$. For $i\leq d-1$, there is an isomorphism:
\begin{align*}
    \text{gr}^a K(X)^\varepsilon\oplus
    \bigoplus_{0\leq j\leq\lfloor \frac{i+d}{2}\rfloor-1}\text{gr}^{a-2-2j} K_\cdot(Z) &\cong P^{\leq i}_f\text{gr}^a K_\cdot(Y)\\
    \big(x,z_0,\cdots,z_{\lfloor \frac{i+d}{2}\rfloor-1}\big)&\mapsto f^*(x)+\sum_{j\leq \lfloor \frac{i+d}{2}\rfloor-1} \iota_*\left(\hbar^jq^*(z_j)\right).
\end{align*}
Here $\varepsilon$ is $0$ if $i<0$ and is $1$ otherwise.
This follows from the computation in Subsection \ref{projective} and Proposition \ref{blowup}.

One can check that in the above examples, we have that $\P^{\leq \cdot}_f=P^{\leq \cdot}_f$.

\subsection{Compatibility with the perverse filtration in cohomology}

Consider a proper map $f:X\to S$ with $X$ smooth. 
Define filtrations
$P'^{\leq i}_{f}, P^{\leq i}_{f}, \textbf{P}^{\leq i}_{f}$ on $H^\cdot(X), H^\cdot(X)_{\text{alg}}$ as in Subsections \ref{Kper} and \ref{Kper2}.
We have that 
\[\text{image}\left(\mathfrak{c}: P^{\leq i}_f\text{gr}^j K_0(X)_\mathbb{Q}\to P^{\leq i}_f H^{2j}(X)\right)=P^{\leq i}_f H^{2j}(X)_{\text{alg}}.\]
We use the notation ${}^pH^{\leq i}_f(X)_{\text{full}}$ for the cohomology of summands of ${}^p\tau^{\leq i}Rf_*IC_X$ with support $S$.

\begin{prop}\label{cyclemap}
There exist natural inclusions
\begin{align*}
\P^{\leq i}_fH^\cdot(X)\subset P^{\leq i}_f H^\cdot(X)&\subset {}^pH^{\leq i}_f(X)\\
\P^{\leq i}_fH^\cdot(X)_{\text{alg}}\subset P^{\leq i}_f H^\cdot(X)_{\text{alg}}&\subset {}^pH^{\leq i}_f(X)_{\text{alg}}.
\end{align*}
    Thus the cycle map restricts to
\begin{align*}
    \mathfrak{c}: P^{\leq i}_f \text{gr}^\cdot K_0(X)_{\mathbb{Q}}&\to {}^pH^{\leq i}_f(X)_{\text{alg}}\\
    \mathfrak{c}: \P^{\leq i}_f \text{gr}^\cdot K_0(X)_{\mathbb{Q}}&\to {}^pH^{\leq i}_f(X)_{\text{alg}}.
    \end{align*}
\end{prop}



\begin{proof}

Let $\pi:T\to S$ be a generically finite map with $T$ smooth. Consider a correspondence 
\[\Gamma\in \text{gr}_{\dim X-s}K_{T\times_S X, 0}(T\times X)\] such that 
\[\left\lfloor \frac{i+\dim X-\dim T}{2}\right\rfloor\geq s.\] 
The correspondence $\Gamma$ induces a map of constructible sheaves on $S$:
\begin{align*}\label{boundd}
R\pi_*\mathbb{Q}_T[-2s]&\xrightarrow{\Phi_\Gamma} Rf_*\mathbb{Q}_X.\\
Rp_*IC_{T}[\dim X-\dim T-2s]&\xrightarrow{\Phi_\Gamma} Rf_*IC_X.
\end{align*}
If $\pi$ is not surjective, $R\pi_*IC_T$ has summands with support $W\subsetneqq S$. If $\pi$ is surjective,
the complex $R\pi_*IC_{T}$ has summands $IC_S(\mathcal{L})$ of full support and of perverse degree zero, and other summands with support $W\subsetneqq S$. The perverse degree of the sheaf with support $S$ in the image of $\Phi_\Gamma$ is
\[\dim X-\dim T-2s\leq i.\] Thus $P'^{\leq i}_fH^\cdot(X)$ contains cohomology of sheaves $IC_S(\mathcal{L})[j]$ with $j\leq i$ which appear as summands of $Rf_*IC_X$ and of other sheaves with support $W\subsetneqq S$. Thus \[P'^{>i}_fH^\cdot(X)\twoheadrightarrow {}^pH^{>i}_f(X)_{\text{full}}.\]
Using the notation in Subsection \ref{filtPP}, we have that
\[P^{\leq i}_fH^\cdot(X):=\bigcap_{V\subsetneqq S}\bigcap_{a\in A_V}\text{ker}\left(\tau^{a*}_V:P'^{\leq i}_fH^\cdot(X)\to P'^{> i+c^a_V}_{\widetilde{f^a_V}}H^\cdot\left(\widetilde{X^a_V}\right)\right).\]
In particular, 
\[P^{\leq i}_fH^\cdot(X)\subset\bigcap_{V\subsetneqq S}\bigcap_{a\in A_V}\text{ker}\left(\tau^{a*}_V:P'^{\leq i}_fH^\cdot(X)\to {}^pH^{> i+c^a_V}_{\widetilde{f^a_V}}\left(\widetilde{X^a_V}\right)_{\text{full}}\right).\]
Using \eqref{comp2}, we obtain that 
$P^{\leq i}_fH^\cdot(X)\subset {}^pH^{\leq i}_f(X)$.

\end{proof}

\textbf{Remark.} We expect equalities $\P^{\leq i}_fH^\cdot(X)_{\text{alg}}= P^{\leq i}_f H^\cdot(X)_{\text{alg}}= {}^pH^{\leq i}_f(X)_{\text{alg}}$ in the above Proposition.

\section{Intersection $K$-theory}\label{4}

\subsection{Definition of intersection $K$-theory}\label{defintK}
Let $S$ be a variety, let $U$ be an open subset, let $f:X\to S$ be such $f^{-1}(U)\to U$ is étale, and let $L=f_*\left(\mathbb{Z}_{f^{-1}(U)}\right)$.
Recall the notation of Subsection \ref{filtPP}. Define 
\begin{align*}
\widetilde{P}^{\leq i}_f\KK&:=
\text{image}\left(\bigoplus_{V\subsetneqq S}\bigoplus_{a\in A_V}P^{\leq i}_f\text{gr}^\cdot K_{X^a_V}(X)\to P^{\leq i}_f\KK\right)\\
\widetilde{\P}^{\leq i}_f\KK&:=
\text{image}\left(\bigoplus_{V\subsetneqq S}\bigoplus_{a\in A_V}\P^{\leq i}_f\text{gr}^\cdot K_{X^a_V}(X)\to \P^{\leq i}_f\KK\right).
\end{align*}
Define
\begin{align*}
\text{gr}^\cdot IK_\cdot(S,L)&:=P^{\leq 0}_f\text{gr}^\cdot K_\cdot(X)\big/\left(\widetilde{P}^{\leq 0}_f\text{gr}^\cdot K_\cdot(X)\cap \text{ker}\,f_*\right)\\
\text{gr}^\cdot \I K_\cdot(S,L)&:=\P^{\leq 0}_f\text{gr}^\cdot K_\cdot(X)\big/\left(\widetilde{\P}^{\leq 0}_f\text{gr}^\cdot K_\cdot(X)\cap\text{ker}\,f_*\right).
\end{align*}

\begin{thm}\label{intK}
The definitions of $\text{gr}^\cdot IK_{\cdot}(S,L)$ and $\text{gr}^\cdot \I K_{\cdot}(S,L)$ do not depend on the choice of the map $f:X\to S$ with $f^{-1}(U)\to U$ étale such that $L\cong f_*\left(\mathbb{Z}_{f^{-1}(U)}\right)$. Further, let $U^o\subset U$ be an open set and let $L^o:=L|_{U^o}$. Then 
\begin{align*}
    \text{gr}^\cdot IK_\cdot(S,L)&\cong \text{gr}^\cdot IK_\cdot(S,L^o),\\
    \text{gr}^\cdot \I K_\cdot(S,L)&\cong\text{gr}^\cdot \I K_\cdot(S,L^o).
\end{align*}
\end{thm}

We start with some preliminary results.  
Let $f:X\to S$ be a proper map with $X$ smooth. Let $Z$ be a smooth subvariety of $X$ with normal bundle $N$, $Y=\text{Bl}_ZX$, and $E=\mathbb{P}_Z(N)$ the exceptional divisor
\begin{equation*}
    \begin{tikzcd}
    E\arrow[r, hook, "\iota"]\arrow[d,"p"]& Y\arrow[d,"\pi"]\\
    Z\arrow[r, hook, "j"]& X.
    \end{tikzcd}
\end{equation*}
Consider the proper maps
\begin{equation*}
    \begin{tikzcd}
    E\arrow[r, hook, "\iota"]\arrow[dr,"h"']& Y\arrow[d,"g"]\arrow[r,"\pi"]&X\arrow[dl,"f"]\\
    & S.&
    \end{tikzcd}
\end{equation*}
Let $X'\hookrightarrow X$ be a closed subset, and denote its preimages in $Y$, $Z$, $E$ by $Y'$, $Z'$, $E'$ respectively. Denote by \[\text{gr}_\cdot K_{Y'}(Y)^0=\text{ker}\big(\pi_*: \text{gr}_\cdot K_{Y'}(Y)\to \text{gr}_\cdot K_{X'}(X)\big).\]

\begin{prop}\label{prp}
 Let $T\to S$ be a map with $T$ smooth which is generically finite onto its image. Then
\begin{align*}
    \text{gr}_\cdot K_{T\times_SY'}(T\times Y)&=\pi^*\text{gr}_\cdot K_{T\times_SX'}(T\times X)\oplus \text{gr}_\cdot K_{T\times_SE'}(T\times Y)^0\\
    \text{gr}_\cdot K^q_{T\times_SY'}(T\times Y)&=\pi^*\text{gr}_\cdot K^q_{T\times_SX'}(T\times X)\oplus \text{gr}_\cdot K^q_{T\times_SE'}(T\times Y)^0.
\end{align*}
\end{prop}

\begin{proof}
Let $c+1$ be the codimension of $Z$ in $X$. Denote by $\mathcal{O}(1)$ the canonical line bundle on $E$ and let $\hbar=c_1(\mathcal{O}(1))\in\text{gr}^2K_0(E)$.
There is a semi-orthogonal decomposition \cite[Theorem 4.2]{BO} with $\pi^*$ fully faithful on $D^b(X)$:
\[D^b(Y)=\left\langle \pi^*D^b(X),\iota_*\left(p^*D^b(Z)\otimes \OO(-1)\right),\cdots,\iota_*\left(p^*D^b(Z)\otimes \OO(-c)\right)\right\rangle,\]
which implies that 
\[\text{gr}^j K_\cdot(Y)=\pi^*\text{gr}^j K_\cdot(X)\oplus \bigoplus_{0\leq k\leq c-1}\iota_*\left(\hbar^k\cdot p^*\text{gr}^{j-2-2k} K_\cdot(Z)\right).\]
Using the analogous decomposition for $Y\setminus Y'=\text{Bl}_{Z\setminus Z'}\left(X\setminus X'\right)$ and the localization sequence in K-theory \cite[V.2.6.2]{W}, we obtain that
\begin{equation*}
\text{gr}^j K_{Y'}(Y)=\pi^*\text{gr}^j K_{X'}(X)\oplus \bigoplus_{0\leq k\leq c-1}\iota_*\left(\hbar^k\cdot p^*\text{gr}^{j-2-2k} K_{Z'}(Z)\right).
\end{equation*}
In particular, we have that
\begin{equation*}
\text{gr}^j K_{T\times_S Y'}(T\times Y)=\pi^*\text{gr}^j K_{T\times_S X'}(T\times X)\oplus \bigoplus_{0\leq k\leq c-1}\iota_*\left(\hbar^k\cdot p^*\text{gr}^{j-2-2k} K_{T\times_S Z'}(T\times Z)\right)
\end{equation*}
and thus that 
\[\text{gr}_\cdot K_{T\times_SY'}(T\times Y)=\pi^*\text{gr}_\cdot K_{T\times_SX'}(T\times X)\oplus \text{gr}_\cdot K_{T\times_SE'}(T\times Y)^0.\]
By Proposition \ref{quasi}, we also have that 
\[\text{gr}_\cdot K^q_{T\times_SY'}(T\times Y)=\pi^*\text{gr}_\cdot K^q_{T\times_SX'}(T\times X)\oplus \text{gr}_\cdot K^q_{T\times_SE'}(T\times Y)^0.\]
\end{proof}

An immediate corollary of Proposition \ref{prp} is:

\begin{cor}\label{blowup2}
We continue with the notation from Proposition \ref{prp}. There are decompositions
\begin{align*}
    P'^{\leq i}_g \text{gr}^\cdot K_{Y'}(Y)&=\pi^*P'^{\leq i}_f\text{gr}^\cdot K_{X'}(X)\oplus P'^{\leq i}_g \text{gr}^\cdot K_{E'}(Y)\\
    \P'^{\leq i}_g \text{gr}^\cdot K_{Y'}(Y)&=\pi^*\P'^{\leq i}_f\text{gr}^\cdot K_{X'}(X)\oplus \P'^{\leq i}_g \text{gr}^\cdot K_{E'}(Y).
\end{align*}
\end{cor}

We next prove:

\begin{prop}\label{blowup}
We continue with the notation from Proposition \ref{prp}. There are decompositions
\begin{align*}
    P^{\leq i}_g \text{gr}^\cdot K_{Y'}(Y)&=\pi^*P^{\leq i}_f\text{gr}^\cdot K_{X'}(X)\oplus P^{\leq i}_g \text{gr}^\cdot K_{E'}(Y)\\
    \P^{\leq i}_g \text{gr}^\cdot K_{Y'}(Y)&=\pi^*\P^{\leq i}_f\text{gr}^\cdot K_{X'}(X)\oplus \P^{\leq i}_g \text{gr}^\cdot K_{E'}(Y).
\end{align*}
\end{prop}

\begin{proof}
We use the notation from Subsection \ref{filtPP}. For $V\subsetneqq S$, let $A_V$ be the set of irreducible components of $f^{-1}(V)$. Let $X^a_V$ be such a component.

If $X^a_V\subset Z$, then there is only one irreducible component $Y^a_V=\mathbb{P}_{X^a_V}(N)$ of $g^{-1}(V)$ above it. 

If $X^a_V$ is not in $Z$, then there is one component $Y^a_V$ of $g^{-1}(V)$ birational to $X^a_V$. The other components are $\mathbb{P}_{W^b_V}(N)$, where $W^b_V$ is an irreducible component of $X^a_V\cap Z$. Denote by $B_a$ the set of such components.   
For $a\in A$ and $b\in B_a$, consider resolutions of singularities $r$ such that
\begin{equation*}
    \begin{tikzcd}
    \widetilde{Y^a_V}\arrow[d]\arrow[r, "r"]& Y^a_V\arrow[d]\arrow[r]&Y\arrow[d]&\\
    \widetilde{X^a_V}\arrow[r, "r"]&X^a_V\arrow[r]&X\\
     &X^a_V\cap Z\arrow[u, hook]\\
     \widetilde{W^b_V}\arrow[r, "r"]\arrow[uu]&W^b_V.\arrow[u, hook]
    \end{tikzcd}
\end{equation*}
Denote by $\tau$ maps as in Subsection \ref{filtPP}, for example $\tau^a_V:\widetilde{X^a_V}\to X$, and by $\mu$ the map
\begin{equation}\label{mumaps}
    \tau^b_V: \widetilde{W^b_V}\xrightarrow{\mu}\widetilde{X^a_V}\xrightarrow{\tau^a_V}X.
\end{equation}
We consider the proper maps
\begin{align*}
\widetilde{f^a_V}&: \widetilde{X^a_V}\to X^a_V\to V\\
\widetilde{g^a_V}&: \widetilde{Y^a_V}\to Y^a_V\to V\\
\widetilde{f^b_V}&: \widetilde{W^b_V}\to W^b_V\to V\\
\widetilde{g^b_V}&: \mathbb{P}_{\widetilde{W^b_V}}(N)\to
\mathbb{P}_{W^b_V}(N)\to V.
\end{align*}
Denote by 
\begin{align*}
    &c^a_V=\text{codim}\left(X^a_V\text{ in }X\right)=\text{codim}\left(Y^a_V\text{ in }Y\right)\\
    &c^b_V=\text{codim}\left(W^b_V\text{ in }X\right)\\
    &c'^b_V=\text{codim}\left(\mathbb{P}_{W^b_V}(N)\text{ in }Y\right)
\end{align*} the codimensions as in Subsection \ref{filtPP}. By Proposition \ref{smpr}, we have that
\begin{multline}\label{onee}
\text{ker}\left(\tau^{a*}_V:\pi^*P'^{\leq i}_f\text{gr}^\cdot K_\cdot(X)\to P'^{>i+c^a_V}_{\widetilde{g^a_V}}\text{gr}^\cdot K_\cdot\left(\widetilde{Y^a_V}\right)\right)\cong\\
\text{ker}\left(\tau^{a*}_V:P'^{\leq i}_f\text{gr}^\cdot K_\cdot(X)\to P'^{>i+c^a_V}_{\widetilde{f^a_V}}\text{gr}^\cdot K_\cdot\left(\widetilde{X^a_V}\right)\right).
\end{multline}
By Proposition \ref{smpr} and Proposition \ref{pullback} for the map $\mu$ in \eqref{mumaps}, we have that
\begin{multline}\label{twoo}
\text{ker}\left(\tau^{b*}_V:\pi^*P'^{\leq i}_f\text{gr}^\cdot K_\cdot(X)\to P'^{>i+c'^b_V}_{\widetilde{g^b_V}}\text{gr}^\cdot K_\cdot\left(\mathbb{P}_{\widetilde{W^b_V}}(N)\right)\right)\cong\\
\text{ker}\left(\tau^{b*}_V:P'^{\leq i}_f\text{gr}^\cdot K_\cdot(X)\to P'^{>i+c^b_V}_{\widetilde{f^b_V}}\text{gr}^\cdot K_\cdot\left(\widetilde{W^b_V}\right)\right)\supset\\
\text{ker}\left(\tau^{a*}_V:P'^{\leq i}_f\text{gr}^\cdot K_\cdot(X)\to P'^{>i+c^a_V}_{\widetilde{f^a_V}}\text{gr}^\cdot K_\cdot\left(\widetilde{X^a_V}\right)\right).
\end{multline}
Let $B_V$ be the set of irreducible components of $g^{-1}(V)$. For $d\in B_V$, denote by $\widetilde{g^d_V}:\widetilde{Y^d_V}\to V$ and
let $c^d_V:=\text{codim}\left(Y^d_V\text{ in }Y\right)$. We have that $B_V=A\cup \bigcup_{a\in A}B_a.$
The statements in \eqref{onee} and \eqref{twoo} imply that
\begin{align*}
\pi_*&: \bigcap_{V\subsetneqq S}\bigcap_{d\in B_V}\text{ker}\left(\tau^{d*}_V:\pi^*P'^{\leq i}_f\text{gr}^\cdot K_\cdot(X)\to P'^{>i+c^d_V}_{\widetilde{g^d_V}}\text{gr}^\cdot K_\cdot\left(\widetilde{Y^d_V}\right)\right)\cong\\
&\bigcap_{V\subsetneqq S}\bigcap_{a\in A_V}\text{ker}\left(\tau^{a*}_V:P'^{\leq i}_f\text{gr}^\cdot K_\cdot(X)\to P'^{>i+c^a_V}_{\widetilde{f^a_V}}\text{gr}^\cdot K_\cdot\left(\widetilde{X^a_V}\right)\right).
\end{align*}
Using Corollary \ref{blowup2}, we obtain that 
\[P^{\leq i}_{g}\, \text{gr}^\cdot K_{Y'}(Y)=\pi^*P^{\leq i}_{f}\,\text{gr}^\cdot K_{X'}(X)\oplus P^{\leq i}_{g}\,\text{gr}^{\cdot} K_{E'}(Y)^0.\]
The analogous statement for $\P^{\leq i}$ follows similarly.
\end{proof}

\begin{proof}[Proof of Theorem \ref{intK}]
Any two such varieties $f:X\to S$ and $f':X'\to S$ are birational, so by \cite{akmw} there is a smooth variety $W$ such that 
\begin{equation*}
    \begin{tikzcd}
     & W\arrow[dl,"\pi"']\arrow[dr,"\pi'"]& \\
     X\arrow[dr,"f"']& &X'\arrow[dl,"f'"]\\
      &S&
    \end{tikzcd}
\end{equation*}
and the maps $\pi$ and $\pi'$ are successive blow-ups along smooth subvarieties of $X$ and $X'$, respectively. It thus suffices to show that
\begin{equation}\label{eqq}
P^{\leq 0}_f\text{gr}^\cdot K_\cdot(X)\Big/\left(\widetilde{P}^{\leq 0}_f\text{gr}^\cdot K_\cdot(X)\cap \text{ker}\,f_*\right)
\cong
P^{\leq 0}_g\text{gr}^\cdot K_\cdot(Y)\Big/\left(\widetilde{P}^{\leq 0}_g\text{gr}^\cdot K_\cdot(Y)\cap \text{ker}\,g_*\right),
\end{equation}
where $\pi:Y\to X$ is the blow up along smooth subvariety $Z\hookrightarrow X$ and
\begin{equation*}
    \begin{tikzcd}
    Y\arrow[dr, "g"']\arrow[r,"\pi"]& X\arrow[d,"f"]\\
    & S.
    \end{tikzcd}
\end{equation*}
By Proposition \ref{blowup}, we have that
\begin{align*}
    P^{\leq i}_g \text{gr}^\cdot K_\cdot(Y) & =\pi^*P^{\leq i}_f\text{gr}^\cdot K_\cdot(X)\oplus P^{\leq i}_g\text{gr}^\cdot K_E(Y)^0\\
     \widetilde{P}^{\leq i}_g \text{gr}^\cdot K_\cdot(Y) & =\pi^*\widetilde{P}^{\leq i}_f\text{gr}^\cdot K_\cdot(X)\oplus P^{\leq i}_g\text{gr}^\cdot K_E(Y)^0.
\end{align*}
The map $\pi^*$ is injective.
Taking the quotients we thus obtain the isomorphism \eqref{eqq}. The analogous statement for $\I K$ is similar.
\end{proof}

\subsection{Cycle map for intersection $K$-theory}\label{cyclemap2}

Let $S$ be a variety, let $U$ be an open subset, let $f:X\to S$ be such $f^{-1}(U)\to U$ is étale, and let $L=f_*\left(\mathbb{Z}_{f^{-1}(U)}\right)$.

\begin{prop}
The cycle map $\text{ch}:\text{gr}^j K_0(X)_\mathbb{Q}\to H^{2j}(X)$ induces cycle maps independent of the map $f:X\to S$ as in Subsection \ref{defintK}:
\begin{align*}
\mathfrak{c}&: \text{gr}^j IK_0(S,L)_\mathbb{Q}\to IH^{2j}(S,L\otimes \mathbb{Q})\\
\mathfrak{c}&: \text{gr}^j \I K_0(S,L)_\mathbb{Q}\to IH^{2j}(S,L\otimes \mathbb{Q}).
\end{align*}
\end{prop}

\begin{proof}
Define $P'^{\leq i}_fH^\cdot_{X^a_V}(X)$ as in Subsection \ref{Kper} and denote by 
\[\widetilde{P}^{\leq 0}_fH^\cdot(X):=
\text{image}\left(\bigoplus_{V\subsetneqq S}\bigoplus_{a\in A_T}P'^{\leq i}_f H^\cdot_{X^a_V}(X)\to H^\cdot(X)\right)\cap P^{\leq 0}_fH^\cdot(X).\]
Denote by ${}^p\widetilde{H}^{\leq 0}_f(X)\subset {}^pH^{\leq 0}_f(X)$ the sum of summands of ${}^p\tau^{\leq 0}Rf_*IC_X$ with support strictly smaller than $S$. 
By Proposition \ref{cyclemap}, the cycle map respects the perverse filtrations in $K$-theory and cohomology
\begin{align*}
\mathfrak{c}&:P^{\leq 0}_f\text{gr}^jK_0(X)_\mathbb{Q}\to P^{\leq 0}_fH^{2j}(X)\hookrightarrow {}^pH^{\leq 0}_f(X)\\
\mathfrak{c}&:\widetilde{P}^{\leq 0}_f\text{gr}^jK_0(X)_\mathbb{Q}\to \widetilde{P}^{\leq 0}_fH^{2j}(X)\hookrightarrow {}^p\widetilde{H}^{\leq 0}_f(X).
\end{align*}
Taking the quotient and using \eqref{comp1}, we obtain a map
\[\mathfrak{c}: \text{gr}^j IK_0(S,L)_\mathbb{Q}\to IH^{2j}(S,L\otimes\mathbb{Q}).\]
The proof that the above cycle map is independent of the map $f$ chosen follows as in Theorem \ref{intK}. The argument for $\I K$ is similar.
\end{proof}

\subsection{Further properties of intersection $K$-theory}\label{multK}

Intersection cohomology satisfies the following properties, the second one explaining its name \cite[Motivation]{CH1}:
\begin{itemize}
\item The natural map $H^i(S)\to H^{\text{BM}}_{2d-i}(S)$ factors through \[H^i(S)\to IH^i(S)\to H^{\text{BM}}_{2d-i}(S).\]


\item There is a natural intersection map \[IH^i(S)\otimes IH^j(S)\to H^{\text{BM}}_{2d-i-j}(S)\] which is non-degenerate for $i+j=2d$.
\end{itemize}

We prove analogous, but weaker versions of the above properties in $K$-theory. 


\begin{prop}\label{propIK}
(a) There are natural maps
\begin{align*}
\text{gr}^i IK_\cdot(S)&\to \text{gr}_{d-i} G_\cdot(S)\\
\text{gr}^i \I K_\cdot(S)&\to \text{gr}_{d-i} G_\cdot(S).
\end{align*}


(b) There are natural intersection maps \begin{align*}
\text{gr}^i IK_\cdot(S)\otimes \text{gr}^j IK_\cdot(S)&\to \text{gr}_{d-i-j} G_\cdot(S)\\
\text{gr}^i \I K_\cdot(S)\otimes \text{gr}^j \I K_\cdot(S)&\to \text{gr}_{d-i-j} G_\cdot(S).
\end{align*}
\end{prop}



\begin{proof}
Let $f:X\to S$ be a resolution of singularities. We discuss the claims for $IK_\cdot$, the ones for $\I K_\cdot$ are similar. We construct maps as above using $f$. They are independent by $f$ by an argument as in Theorem \ref{intK}. 

(a) There is a natural map
$\text{gr}^i K_\cdot(X)=\text{gr}_{d-i}G_\cdot(X)\xrightarrow{f_*}\text{gr}_{d-i}G_\cdot(S),$ and we thus obtain a map
\[\text{gr}^iIK_\cdot(S)=
P^{\leq 0}_f\text{gr}^iK_\cdot(X)\big/ \left(\widetilde{P}^{\leq 0}_f\text{gr}^iK_\cdot(X)\cap \text{ker}\,f_*\right)\to \text{gr}_{d-i}G_\cdot(S).\]

(b) Consider the composite map
\[P^{\leq 0}_f\text{gr}^i K_\cdot(X)\boxtimes P^{\leq 0}_f\text{gr}^j K_\cdot(X)\to \text{gr}^{i+j}K_\cdot(X\times X)\xrightarrow{\Delta^*}\text{gr}^{i+j}K_\cdot(X)\xrightarrow{f_*}\text{gr}_{d-i-j}G_\cdot(S).\]
The subspaces \begin{align*}
&\left(\widetilde{P}^{\leq 0}_f\text{gr}^iK_\cdot(X)\cap \text{ker}\,f_*\right)\boxtimes P^{\leq 0}_f\text{gr}^iK_\cdot(X)\\
&P^{\leq 0}_f\text{gr}^iK_\cdot(X)\boxtimes \left(\widetilde{P}^{\leq 0}_f\text{gr}^iK_\cdot(X)\cap \text{ker}\,f_*\right)
\end{align*}
are in the kernel of $f_*\Delta^*=\Delta^*\left(f_*\boxtimes f_*\right)$. We thus obtain the desired map. 
\end{proof}

\subsection{Computations of intersection $K$-theory}\label{computations}
\subsubsection{} If $S$ is smooth, then $\text{gr}^\cdot IK_\cdot(S)=\text{gr}^\cdot \I K_\cdot(S)=
\text{gr}^\cdot K_\cdot(S).$

\subsubsection{} 
Let $f:X\to S$ be a small resolution of singularities. Then
\[\text{gr}^\cdot \I K_0(S)=\text{gr}^\cdot K_0(X).\]

Let $T\xrightarrow{\pi}S$ be a generically surjective finite map from $T$ smooth. By Proposition \ref{lb}, $\text{gr}_{\dim X}K^q_{T\times_SX}(T\times X)$ is generated by the irreducible components of $T\times_SX$ dominant over $S$. This means that the cycles in $\text{gr}_{a}K^q_{T\times_SX}(T\times X)$ supported on the exceptional locus have $a<\dim X$, and thus they only contribute in perverse degrees $\geq 1$, see \eqref{ranges}.

Next, say that $T\xrightarrow{\pi}S$ has image $V\subsetneqq S$. Let $[\Gamma]\in \text{gr}_{\dim X-a}K^q_{T\times_SX}(T\times X)$. By Proposition \ref{lb}, $a\leq \dim X-\dim X_V$. 
If it contributes in perverse degree $i$, then
\[\left\lfloor  \frac{i+\dim X-\dim V}{2}\right\rfloor \geq \dim X-\dim X_V,\]
and thus that 
\[i\geq \dim X+\dim V-2\dim X_V\geq 1.\]
This means that $\widetilde{\P}^{\leq 0}_f\KK=0$. By Theorem \ref{sd}, $\P^{\leq 0}_f\text{gr}^\cdot K_0(X)=\text{gr}^\cdot K_0(X)$, and thus $\text{gr}^\cdot \I K_0(S)=\text{gr}^\cdot K_0(X).$

\subsubsection{}
Let $S$ be a surface. Consider a resolution of singularities $f:X\to S$. Let $B$ be the set of singular points of $S$. For each $p$ in $B$, let 
$A_p=\{C_p^a\}$ be the set of irreducible components of $X_p:=f^{-1}(p)$. For each such curve, consider the diagram
\begin{equation*}
    \begin{tikzcd}
    C_p^a\arrow[d,"h^a_p"']\arrow[r, hook, "g^a_p"]&X\arrow[d,"f"]\\
    p\arrow[r, hook]& S.
    \end{tikzcd}
\end{equation*}
Consider the maps 
\begin{align*}
    &m^a_p:=g^a_{p*}h^{a*}_p:K_\cdot(p)\to \text{gr}^1K_\cdot(X)\\
    &\Delta^a_p:=h^a_{p*}g^{a*}:\text{gr}^1K_\cdot(X)\to K_\cdot(p).
\end{align*}
We claim that
\[\widetilde{\P}^{\leq 0}_f\KK=
\text{image}\left(\bigoplus_{p\in B}\bigoplus_{a\in A_p} m^a_p: K_\cdot(p)\to \text{gr}^1K_\cdot(X)
\right).\] 

The correspondences which contribute to $\widetilde{\P}^{\leq 0}_f$ are in $\text{gr}_{2-s} K^q_{T\times_SX}(T\times X)$ for $\pi: T\to S$ a generically finite map onto its image $V\subsetneqq S$ with $T$ smooth. 
By Proposition \ref{lb},
\[\left\lfloor \frac{2-\dim V}{2}\right\rfloor\geq s\geq \dim X-\dim X_V.\]
So the map $T\to S$ is the inclusion of a point $p\hookrightarrow S$ for $p\in B$ and $\Gamma$ is in  $\text{gr}_1G_{X_p}(X)$. Further, for $p, q\in B$, $a\in A_p$, $b\in A_q$:
\[\Delta^b_qm^a_p=\delta_{pq}\delta_{ab}\,\text{id}.\]
This means that:
\[\bigoplus_{p\in B}\bigoplus_{a\in A_p}m^a_{p}:\bigoplus_{p\in B} K_\cdot(p)^{|A_p|}\cong \widetilde{\P}^{\leq 0}_f\text{gr}^1K_\cdot(X).\]
The map $f$ is semismall, so by Theorem \ref{sd} we obtain a form of the decomposition theorem for the map $f$:
\[\text{gr}^\cdot K_0(X)\cong\text{gr}^\cdot \I K_0(S)\oplus\bigoplus_{p\in B} K_0(p)^{|A_p|}.\]
See Section \ref{semismall} for further discussions of the decomposition theorem for semismall maps.

\subsubsection{}\label{surface}
Let $Y$ be a smooth projective variety of dimension $d$ and let $\mathcal{L}$ be a line bundle on $Y$. Consider the cone $S=C_Y\mathcal{L}$ and its resolution of singularities
\[X:=\text{Tot}_Y\mathcal{L}\xrightarrow{f} S.\]
Let $o$ be the vertex of the cone $X$. There is only one fiber with nonzero dimension
\begin{equation*}
    \begin{tikzcd}
    Y\arrow[r, hook, "\iota"]\arrow[d,"g"]&X\arrow[d,"f"]\\
    o\arrow[r, hook]&S.
    \end{tikzcd}
\end{equation*}
Using the correspondence $X\cong\Delta\hookrightarrow X\times_SX$, we see that
\[P'^{\leq 0}_f\KK=\KK.\]

For $V\subsetneqq S$, the irreducible components of $f^{-1}(V)$ are $f_V:W\to V$ birational to $V$ and, if $V$ contains $o$, the fiber $Y$. As above, we have that $P'^{\leq 0}_{f_V}\text{gr}^\cdot G_\cdot(W)=\text{gr}^\cdot G_\cdot(W)$, so the conditions in defining $P^{\leq i}_f$ are automatically satisfied for these irreducible components $W$. We thus have that
\[P^{\leq 0}_f\text{gr}^\cdot K_\cdot(X)=\text{ker}\left(\iota^*:\text{gr}^\cdot K_\cdot(X)\to P'^{>1}_g\text{gr}^\cdot K_\cdot(Y) \right).\]
By the computation in Subsection \ref{oneex}, \[P'^{>1}_g\text{gr}^j K_\cdot(Y)=
\begin{cases} \mbox{$\text{gr}^jK_\cdot(Y)$} & \mbox{if } j> \lfloor \frac{d+1}{2} \rfloor, \\ \mbox{$0$} & \mbox{otherwise.} \end{cases}\]
The map $\iota^*: \text{gr}^jK_\cdot(X)\to \text{gr}^jK_\cdot(Y)$ is an isomorphism, so we have that
\[P^{\leq 0}_f\text{gr}^jK_\cdot(X)=
\begin{cases} \mbox{$\text{gr}^jK_\cdot(Y)$} & \mbox{if } j\leq \lfloor \frac{d+1}{2} \rfloor, \\ \mbox{$0$} & \mbox{otherwise.} \end{cases}\]
Further, $\widetilde{P}^{\leq 0}_f\text{gr}^\cdot K_\cdot(X)$ is generated by the cycles over $X_o\cong Y$ of codimension between $0$ and $\lfloor \frac{d-1}{2}\rfloor$. The map 
\[\iota_*:\text{gr}^i K_\cdot(Y)\to \text{gr}^{i+2} K_\cdot(X)\cong \text{gr}^{i+2} K_\cdot(Y)\] is multiplication by the class $\hbar:=c_1(\mathcal{L}|_Y)\in \text{gr}^2 K_0(Y)$.
As a vector space, we thus have that
\[ \text{gr}^j IK_\cdot(S)=
\begin{cases} \mbox{$\text{gr}^j K_\cdot(Y)\big/\hbar\,\text{gr}^{j-2} K_\cdot(Y)$} & \mbox{if } j\leq \lfloor \frac{d+1}{2} \rfloor, \\ \mbox{$0$} & \mbox{otherwise.} \end{cases} \]
The computation in cohomology is similar, see \cite[Example 2.2.1]{dCM2}.

\section{The decomposition theorem for semismall maps}
\label{semismall}

We will be using the notation from Subsection \ref{ssn}. For $a,b \in A$, we write $b<a$ if $S_b\subsetneqq S_a$. Denote by $\iota_{ba}:X_b\hookrightarrow X_a$.
For $a\in A$, define 
\[\widetilde{\P}^{\leq 0}_{f}\text{gr}^\cdot K_{X_a}(X)=\text{image}\left(\bigoplus_{b<a}\iota_{ba*}:\P^{\leq 0}_f\text{gr}^\cdot K_{X_b}(X)\to \P^{\leq 0}_f\text{gr}^\cdot K_{X_a}(X)\right).\]
First, we state a more precise version of Conjecture \ref{dss}. 

\begin{conj}\label{dss2}
Let $f:X\to S$ be a semismall map and consider $\{S_a|\,a\in I\}$ a stratification as in Subsection \ref{ssn}, denote by $A\subset I$ the set of relevant strata. For $a\in A$, consider generically finite maps $\pi_a: T_a\to S_a$ with $T_a$ is smooth such that $\pi_a^{-1}(S_a^o)\to S_a^o$ is smooth and $L_a\cong f_*\left(\mathbb{Z}_{S^o_a}\right)$. For each $a$, there exists a rational map $X_a\dashrightarrow T_a$, and let $\Gamma_a$ be the closure of its graph
\begin{equation*}
    \begin{tikzcd}
    \Gamma_a\arrow[d]\arrow[r]&X_a\arrow[r, hook,"\iota_a"]\arrow[d,"f_a"]&X\arrow[d,"f"]\\
    T_a\arrow[r,"\pi_a"]&S_a\arrow[r, hook]&S,
    \end{tikzcd}
\end{equation*}
The correspondence $\Gamma_a$ induces an isomorphism
\begin{multline}\label{topdegree}
\iota_{a*}\Phi_{\Gamma_a}: \P^{\leq 0}_{\pi_a}\text{gr}^{j-c_a} K_\cdot(T_a)_\mathbb{Q}\big/\widetilde{\P}^{\leq 0}_{\pi_a}\text{gr}^{j-c_a} K_\cdot(T_a)_\mathbb{Q}\cong\\ \iota_{a*}\left(\P^{\leq 0}_{f}\text{gr}^{j} K_{X_a}(X)_\mathbb{Q}\big/\widetilde{\P}^{\leq 0}_{f}\text{gr}^j K_{X_a}(X)_\mathbb{Q}\right)
\end{multline}
and a decomposition
\begin{align*}
\bigoplus_{a\in A} \text{gr}^{j-c_a} \I K_\cdot(S_a, L_a)_\mathbb{Q}&\cong\text{gr}^j K_\cdot(X)_\mathbb{Q}\\
(x_a)_{a\in A}&\mapsto \sum_{a\in A}\iota_{a*}\Phi_{\Gamma_a}(x_a).
\end{align*}
\end{conj}

In relation to \eqref{topdegree}, we propose the following:

\begin{conj}
Let $f:X\to S$ be a surjective map of relative dimension $d$ with $X$ is smooth. Let $U$ be a smooth open subset of $S$ such that $f^{-1}(U)\to U$ is smooth. For $y\in U$, $\pi_1(U, y)$ acts on the irreducible components of $f^{-1}(y)$ of top dimension and let $L$ be the associated local system. Then $L$ is an integer finite local system.
There is an isomorphism
\[\P^{\leq -d}_f\text{gr}^j K_\cdot(X)_\mathbb{Q}\big/\widetilde{\P}^{\leq -d}_f\text{gr}^j K_\cdot(X)_\mathbb{Q}\cong \text{gr}^j \I K_\cdot(S,L)_\mathbb{Q}.\]
\end{conj}

The analogous statement in cohomology follows from the decomposition theorem.
In this section, we prove the following:
\begin{thm}\label{ss}
We use the notation from Conjecture \ref{dss2}. Assume that the maps $\pi_a: T_a\to S_a$ are small. Then Conjecture \ref{dss2} holds for $K_0$.
\end{thm}

We first note a preliminary result.

\begin{prop}\label{localrat}
Consider varieties $S$ and $X$, and a smooth variety $Y$ with surjective maps $f:X\to S$ of relative dimension $d$ and
$g:Y\to S$ of relative dimension $0$.
Assume there exists an open subset $U$ of $S$ and a map $h$ such that:
\begin{equation*}
    \begin{tikzcd}
    g^{-1}(U)\arrow[dr,"g"']& & f^{-1}(U)\arrow[dl,"f"]\arrow[ll,"h"']\\
     &U&
    \end{tikzcd}
\end{equation*}
Denote also by $h$ the rational map $h:X\dashrightarrow Y$. Consider a resolution of singularities $\pi:X'\to X$ such that there exists a regular map $h'$ as follows: 
\begin{equation*}
    \begin{tikzcd}
     &X'\arrow[d,"\pi"]\arrow[ddl, bend right, "h'"']\\
     &X\arrow[d,"f"]\arrow[dl,"h"']\\
     Y\arrow[r,"g"]&S.
    \end{tikzcd}
\end{equation*}
Let $\Gamma$ be the closure of the graph of $h$ in $Y\times X$ and let $\Gamma'$ be the graph of $h'$ in $Y\times X'$. Then the following diagram commutes:
\begin{equation*}
    \begin{tikzcd}
    \text{gr}^\cdot K_\cdot(Y)\arrow[r,"\Phi_{\Gamma'}"]\arrow[dr,"\Phi_\Gamma"']& \text{gr}^\cdot K_\cdot(X')\arrow[d,"\pi_*"]\\
     & \text{gr}^\cdot G_\cdot(X).
    \end{tikzcd}
\end{equation*}
\end{prop}

\begin{proof}
Consider the maps:
\begin{equation*}
    \begin{tikzcd}
     Y\times X' \arrow[d,"\pi'"']\arrow[r,"p'"]&X'\arrow[d,"\pi"]\\
     Y\times X\arrow[r,"p"]\arrow[d,"q"']&X\arrow[d,"f"]\\
     Y\arrow[r,"g"]&S.
    \end{tikzcd}
\end{equation*}
Let $x\in \text{gr}^\cdot K_\cdot(Y)$. We want to show that:
\[\pi_*p'_*(\Gamma'\otimes \pi'^*q^*(x))=p_*(\Gamma\otimes q^*(x)).\]
It suffices to show that
\begin{equation}\label{Gamma}
\pi'_*\left[\Gamma'\right]=\left[\Gamma\right]\text{ in }\text{gr}_\cdot G(X\times Y).
\end{equation}
Both $\Gamma$ and $\Gamma'$ have dimension equal to the dimension of $X$. The map $\pi':\Gamma'\to\Gamma$ is birational, so the cone of 
\[\mathcal{O}_\Gamma\to\pi'_*\mathcal{O}_{\Gamma'}\] is supported on a proper set of $\Gamma$, which implies the equality in \eqref{Gamma}.
\end{proof}

\begin{proof}[Proof of Theorem \ref{ss}]

Let $a\in A$ and consider the diagram:
\begin{equation*}
    \begin{tikzcd}
     & Y_a\arrow[d,"\tau_a"]\arrow[ddl,"h_a"']&\\
     &X_a\arrow[d,"f_a"]\arrow[r, hook]&X\arrow[d,"f"]\\
     T_a\arrow[r,"\pi_a"]&S_a\arrow[r,hook]&S,
    \end{tikzcd}
\end{equation*}
where the map $\tau_a$ is a resolution of singularities. Let $\Gamma_a$ be the closure of the natural rational map $X_a\dashrightarrow T_a$. 
By Proposition \ref{localrat} and Theorem \ref{functorP}, the map $\Phi_{\Gamma_a}$ factors as:
\begin{equation*}
    \Phi_{\Gamma_a}: \text{gr}^j K_\cdot(T_a)\xrightarrow{h_a^*}\text{gr}^j K_\cdot(Y_a)\xrightarrow{\tau_{a*}}\text{gr}^j G_\cdot(X_a)\to \text{gr}^{j+c_a} K_{X_a}(X).
    \end{equation*}
By Theorems \ref{functorialityP} and \ref{sd}, the map $\Phi_{\Gamma_a}$ factors as:
\begin{multline*}
\Phi_{\Gamma_a}: \text{gr}^j K_0(T_a)=\P^{\leq 0}_{h_a}\text{gr}^j K_0(T_a)\xrightarrow{h_a^*}
\P^{\leq -d_a}_{f_a\tau_a}\text{gr}^j K_0(Y_a)\xrightarrow{\tau_{a*}}\P^{\leq -d_a}_{f_a}\text{gr}^j K_0(X_a)\to\\
 \P^{\leq 0}\text{gr}^{j+c_a} K_{X_a,0}(X)\to 
\P^{\leq 0}_f\text{gr}^{j+c_a} K_0(X)=\text{gr}^{j+c_a} K_0(X).
\end{multline*}
We thus obtain a map of vector spaces
\begin{multline}\label{facto}
    \bigoplus_{a\in A}\Phi_{\Gamma_a}: \bigoplus_{a\in A}\text{gr}^{j-c_a} K_0(T_a)\to \bigoplus_{a\in A}\iota_{a*}\left(\P^{\leq 0}_{f}\text{gr}^j K_{X_a,0}(X)\big/\widetilde{\P}^{\leq 0}_{f}\text{gr}^j K_{X_a,0}(X)\right)\\\to \text{gr}^j K_0(X).
\end{multline}
A theorem of de Cataldo--Migliorini \cite[Theorem 4.0.4]{dCM} says that there is an isomorphism:
\[\bigoplus_{a\in A}\Phi_{\Gamma_a}: 
\bigoplus_{a\in A}\text{gr}^{j-c_a} K_0(T_a)_\mathbb{Q}\xrightarrow{\sim} \text{gr}^{j} K_0(X)_\mathbb{Q}.\]
Combining with \eqref{facto}, we see that in this case
\[\Phi_{\Gamma_a}:\text{gr}^{j-c_a} K_0(T_a)_\mathbb{Q}\xrightarrow{\sim} \iota_{a*}\left(\P^{\leq 0}_{f}\text{gr}^j K_{X_a,0}(X)_\mathbb{Q}\big/\widetilde{\P}^{\leq 0}_{f}\text{gr}^j K_{X_a,0}(X)_\mathbb{Q}\right).\]
This implies the claim of Theorem \ref{ss}.

\end{proof}

\end{document}